\documentclass[11pt,a4paper]{article}
\usepackage{amssymb}
\usepackage{amsmath}
\usepackage{amsthm}
\usepackage[all]{xy}
\usepackage{tikz}
\usetikzlibrary{calc}
\usepgflibrary{shapes.geometric}
\usepgflibrary{shapes.misc}
\usetikzlibrary{positioning}
\usetikzlibrary{decorations}
\usetikzlibrary{decorations.pathreplacing}
\usepackage{url}
\usepackage[english]{babel}
\usepackage[T1]{fontenc}

\usepackage[a4paper,left=1in,width=458pt,textheight=650pt,top=1.5in]{geometry}

\usepackage{hyperref}

\hypersetup{
pdftitle={Cluster algebras of infinite rank},
pdfauthor={Jan E. Grabowski and Sira Gratz},
pdfstartview=FitH
}

\theoremstyle{plain}

\theoremstyle{plain}
\theoremstyle{definition}
\newtheorem{lm}{Lemma}[section]
\newtheorem{prop}[lm]{Proposition}
\newtheorem{cor}[lm]{Corollary}
\newtheorem{theorem}[lm]{Theorem}
\newtheorem{defn}[lm]{Definition}

\newtheorem{ex}[lm]{Example}

\theoremstyle{remark}
\newtheorem{rmk}[lm]{Remark}

\theoremstyle{remark}
\newtheorem{remark}[lm]{Remark}

\newcommand{\Hom}{\operatorname{Hom}\nolimits}

\newcommand{\add}{\operatorname{add}\nolimits}

\newcommand{\ZZ}{\mathbb{Z}}

\newcommand{\T}{\mathfrak{t}}

\newcommand{\A}{\mathcal{A}}

\newcommand{\Gr}{\operatorname{Gr}\nolimits}

\DeclareMathOperator{\GrV}{Gr(\mathit{n},\mathit{V})}
\DeclareMathOperator{\Grd}{Gr(\mathit{V}^{\vee},\mathit{n})}
\DeclareMathOperator{\Alg}{Alg}
\DeclareMathOperator{\Set}{Set}

\DeclareMathOperator{\Spec}{Spec}
\DeclareMathOperator{\Pd}{\mathbb{P}(\mathit{V}^{\vee})}
\DeclareMathOperator{\Pdd}{\mathbb{P}(\bigwedge^{\mathit{n}}\mathit{V}^{\vee})}
\DeclareMathOperator{\Sym}{Sym}

\title{Cluster algebras of infinite rank}
\author{Jan E. Grabowski\footnotemark[2] 
\\ \small{\textit{Department of Mathematics and Statistics, Lancaster University,}}
\\ \small{\textit{Lancaster, LA1 4YF, United Kingdom}}
\and Sira Gratz\footnotemark[3]
\\ \small{\textit{Institut f\"{u}r Algebra, Zahlentheorie und Diskrete Mathematik, Leibniz Universit\"{a}t Hannover,}}
\\ \small{\textit{Welfengarten 1, 30167 Hannover, Germany}}
\\
\\ With an appendix by Michael Groechenig\footnotemark[4]
\\ \small{\textit{Mathematical Institute, University of Oxford,}}
\\ \small{\textit{24--29 St. Giles', Oxford, OX1 3PG, United Kingdom}}
\\ \small{and}
\\ \small{\textit{Section de math\'{e}matiques, \'{E}cole Polytechnique F\'{e}d\'{e}rale de Lausanne,}}
\\ \small{\textit{CH-1015 Lausanne, Switzerland}}
}
\date{14th December 2012}

\begin{document}

\maketitle

\renewcommand{\thefootnote}{\fnsymbol{footnote}}
\footnotetext[2]{Email: \url{j.grabowski@lancaster.ac.uk}.  Website: \url{http://www.maths.lancs.ac.uk/~grabowsj/}}
\footnotetext[3]{Email: \url{gratz@math.uni-hannover.de}.}
\footnotetext[4]{Email: \url{groechenig@maths.ox.ac.uk}.}
\renewcommand{\thefootnote}{\arabic{footnote}}
\setcounter{footnote}{0}

\begin{abstract} Holm and J\o rgensen have shown the existence of a cluster structure on a certain category $D$ that shares many properties with finite type $A$ cluster categories and that can be fruitfully considered as an infinite analogue of these.  In this work we determine fully the combinatorics of this cluster structure and show that these are the cluster combinatorics of cluster algebras of infinite rank.  That is, the clusters of these algebras contain infinitely many variables, although one is only permitted to make finite sequences of mutations.

The cluster combinatorics of the category $D$ are described by triangulations of an $\infty$-gon and we see that these have a natural correspondence with the behaviour of Pl\"{u}cker coordinates in the coordinate ring of a doubly-infinite Grassmannian, and hence the latter is where a concrete realization of these cluster algebra structures may be found.  We also give the quantum analogue of these results, generalising work of the first author and Launois.

An appendix by Michael Groechenig provides a construction of the coordinate ring of interest here, generalizing the well-known scheme-theoretic constructions for Grassmannians of finite-dimensional vector spaces.
\end{abstract}

%\tableofcontents

\vfill

\pagebreak

\section{Introduction}

Cluster algebras were introduced by Fomin and Zelevinsky in \cite{FZ1} to provide an algebraic framework for the study of dual canonical bases in coordinate rings of certain algebraic varieties related to semisimple groups, e.g.\ the Grassmann varieties.  The original definition of cluster algebras covers what we will term cluster algebras of finite rank, i.e.\ with clusters consisting of finitely many elements.  Our work here is concerned with an extension of this to cluster algebras of infinite rank, by allowing countable clusters but only finite sequences of mutations. 

Holm and J\o rgensen (\cite{HJ}) have identified a category $D$ that has $\mathbb{Z}A_{+\infty}$ as its Auslander--Reiten quiver and which possesses a cluster structure in the sense of Buan, Iyama, Reiten and Scott (\cite{BIRS}). Furthermore, the cluster tilting subcategories of $D$ have a natural correspondence with certain triangulations of the $\infty$-gon.  The combinatorics of triangulations of the $n$-gon are well-known to model cluster algebras of type $A$ and also the behaviour of Pl\"{u}cker coordinates for Grassmannians of planes and so it is natural to try to associate a cluster algebra to $D$ and to try to relate this to a suitable Grassmannian.  This is indeed what we do.  

Our main goal here is to decategorify the cluster structure on $D$, first by examining the cluster combinatorics associated to triangulations of the $\infty$-gon and then making a link with a suitable algebra having these combinatorics as a cluster structure.  These cluster combinatorics will be of infinite rank and the algebra we consider will be the coordinate ring of an infinite Grassmannian.  Although the theory progresses in a similar fashion to the finite rank case, the infinite setting produces some new and interesting features.

To carry out the programme described above, we adapt the well-known algorithm from the finite rank, type $A_{n}$ situation that associates an exchange quiver with a triangulation of the $\infty$-gon and we show that the resulting quiver coincides with the Gabriel quiver of the corresponding cluster tilting subcategory.  However, unlike the finite rank case, these exchange quivers are not all mutation-equivalent.  This is due to the fact that we are only allowed finite sequences of mutations; Example~\ref{ex:fin-many-mutations} shows that allowing infinite sequences of mutations takes one outside the class of triangulations of the $\infty$-gon, or equivalently away from cluster tilting subcategories.

In fact, there are uncountably many mutation-equivalence classes of exchange quivers of triangulations of the $\infty$-gon and each equivalence class gives rise to a cluster structure on a suitable algebra. This is in stark contrast to the finite type case, where there is only one mutation-equivalence class.  These algebras are certain subalgebras of the coordinate ring $\mathbb{C}[\Gr(2,\pm \infty)]$ of an infinite Grassmannian $\Gr(2,\pm\infty)$ which may be defined in an analogous way to the scheme-theoretic approach to the usual Grassmannians; the technical details of this construction are given in an appendix to this work, written by Michael Groechenig. 

The triangulations that are locally finite give cluster algebra structures on the whole coordinate ring $\mathbb{C}[\Gr(2,\pm \infty)]$ but other triangulations yield cluster structures only on proper subalgebras of this.  As a consequence, we have an infinite rank version of the construction going back to \cite{FZ2} of a cluster algebra structure of type $A_{n}$ on the coordinate ring $\mathbb{C}[\Gr(2,n+3)]$.

The main technical result that we use is a classification of the exchange quivers arising from triangulations, or equivalently of the Gabriel quivers of the (weak) cluster tilting subcategories of $D$.  A precise statement is given in Theorem~\ref{classif-of-exch-quivers} but informally these consist of quivers of type $A_{+\infty}$ or two copies of this, with quivers mutation-equivalent to finite type $A$ quivers attached via oriented 3-cycles, plus possibly a third connected component that is finite and mutation-equivalent to a finite type $A$ quiver.  Consequently we have many different cluster structures of infinite rank---of which some are of type $A_{+\infty}$ but some are not---though all have type $A$ flavour combinatorics.

\vfill
\pagebreak

Independent work on cluster algebras of type $A_{+\infty}$ has been done by Gorsky in \cite{Gorsky}, where he defined a cluster algebra of infinite rank for this type.  Then, as noted above, some of our cluster algebras have this form but others do not.  Our viewpoint is also somewhat different: we study all cluster algebras of infinite rank obtained from the cluster structure of the category $D$.

In \cite{BZ}, Berenstein and Zelevinsky have introduced the notion of a quantum cluster algebra, these being $q$-deformations of classical cluster algebras.  The first author and Launois have shown in \cite{GL} that the quantized coordinate ring over the Grassmannian $\mathbb{C}_q[\Gr(2,n+3)]$ has a quantum cluster structure of type $A_n$.  We extend this result by showing that the triangulations of the $\infty$-gon yield quantum cluster algebra structures on subalgebras of $\mathbb{C}_q[\Gr(2,\pm \infty)]$ analogous to the classical ones described above.

\subsection*{Acknowledgements}

We are indebted to Michael Groechenig for writing the aforementioned appendix to this work. The authors would also like to thank Karin Baur, Thorsten Holm, Peter J\o rgensen, Robert Marsh and Geordie Williamson for helpful conversations.  We would also like to gratefully acknowledge funding from a Lancaster University Research Incentivisation Grant and from the Deutsche Forschungsgemeinschaft. 

\section{Preliminaries} \label{preliminaries}

\subsection{Cluster algebras}

A cluster algebra (of geometric type) is a subring of a field of rational functions and is generated by the so-called cluster variables. The cluster variables are constructed by successive mutations of seeds, where a seed consists of an $n$-set of cluster variables and certain combinatorial information encoded in a quiver. A thorough introduction to cluster algebras can be found in either \cite{Ke10} or \cite{GSV} and we will assume that the reader is familiar with the concept of quiver mutation therein.  For notational purposes, we recall the definition of the cluster algebra associated to a finite (ice) quiver.

\begin{defn}\label{cluster algebra}
Let $Q$ be a finite ice quiver---that is, a quiver with distinguished vertices called frozen---without loops or 2-cycles. Let its set of vertices be $Q_0 = \{1, \ldots, n, n+1, \ldots, n+m\}$ where the vertices $n+1, \ldots, n+m$ are frozen. The full subquiver $Q'$ of $Q$ with vertices $1, \ldots, n$ is called the principal part of $Q$.  The cluster algebra $\A_Q$ is a subalgebra of the field of rational functions $\mathbb{Q}(x_1, \ldots, x_n, x_{n+1}, \ldots, x_{n+m})$ which is defined as follows.

Consider the pair $\Sigma = (Q,\{x_1, \ldots x_n, x_{n+1}, \ldots, x_{n+m}\})$ consisting of the quiver $Q$ and the $(n+m)$-set of indeterminates. We call it the {\em initial seed} of $\A_Q$. From the initial seed, for each of the vertices $1 \leq k \leq n$ in the principal part of the quiver $Q$ we pass to another seed  \[ \mu_k(Q, \{x_1, \ldots, x_n,  x_{n+1}, \ldots, x_{n+m}\}) = (\mu_k(Q), \{x_1 \ldots, x_{k-1}, x'_k, x_{k+1} \ldots, x_{n+m}\}) \] by simultaneously mutating $Q$ at the vertex $k$ and exchanging the variable $x_k$ for $x'_k$ which is defined by the {\em exchange relation}
\begin{eqnarray} \label{exchange relation}
x'_kx_k = \prod_{i \in Q_0}x_i^{\#\{i \to k\}} + \prod_{j \in Q_0} x_j^{\#\{k \to j\}}.
\end{eqnarray}
The empty product is defined to be $1$. 

\vfill
\pagebreak
By iterated mutations at all vertices, we get a collection of seeds. The $(n+m)$-sets in the seeds are called the {\em clusters} and the quivers {\em exchange quivers}. The variables in the union of all clusters are called the {\em cluster variables}, with the cluster variables $x_{n+1}, \ldots, x_{n+m}$ called coefficients. Then the {\em cluster algebra} $\A_Q$ is the subalgebra $\A_Q \subset \mathbb{Q}(x_1, \ldots, x_{n+m})$ generated by the cluster variables.
\end{defn}

\subsection{Quantum cluster algebras and quantum Grassmannians}\label{ss:QCAs}

In \cite{BZ}, Berenstein and Zelevinsky introduced quantum cluster algebras  as a generalization of cluster algebras in the noncommutative setting.  Quantum cluster algebras are noncommutative deformations of classical cluster algebras, obtained by introducing a formal variable $q$ and asking that each quantum cluster is a family of quasi-commuting variables (that is, the variables commute up to powers of $q$). The construction follows an analogous pattern to the classical case: we begin with an initial seed and we obtain a family of seeds by iterated mutation.  Then the variables in all the seeds together generate the quantum cluster algebra. The role of the ambient field $\mathbb{Q}(x_1, \ldots x_n)$ in the classical case is now taken by quantum tori, which are defined as quadratic algebras of the form 
\[ \mathbb{Q}(q^{\frac{1}{2}})[x_1^{\pm 1}, \ldots, x_n^{\pm 1}]/<x_ix_j - q^{L_{ij}}x_jx_i>, \]
where $L$ is an $n\times n$ skew-symmetric integer matrix defining the quasi-commutation rules for the variables $x_i$. We refer to \cite{BZ} for technical details.

Seeds now also include the data of a matrix $L$ as above, describing quasi-commutation relations in the cluster, in addition to the cluster itself and the exchange matrix $B$ that we already have in the classical case.  These quasi-commutation matrices also have a mutation rule and in order for the variables in a mutated cluster to still be quasi-commuting, the matrix $L$ has to satisfy a compatibility condition with the exchange matrix $B$, namely that $(B^TL)_{ij} = d \delta_{ij}$ for some positive integer $d$. 

The quantum exchange relation is similar to the classical one but involves coefficients that are powers of $q$ derived from $B$ and $L$, which we describe now.  Letting $(x_{1},\dotsc ,x_{n})$ be a quantum cluster with associated exchange matrix $B$ and quasi-commutation matrix $L$, the mutation at $x_{k}$ is given by setting $x'_i = x_i$ for $i\neq k$ and defining $x'_k$ by the exchange relation
$$
x'_k = M(-e_k + \sum_{b_{ik}>0}b_{ik}e_i) + M(-e_k - \sum_{b_{ik}<0}b_{ik}e_i),
$$
where
$$
M(a_1, \ldots, a_n) = q^{\frac{1}{2}\sum_{i<j}a_ia_jL_{ji}}x_1^{a_1}\dotsm x_n^{a_n}.
$$
Here $e_{i}$ is the $i$th standard basis vector in $\mathbb{C}^{n}$.  The monomial $M$ (as we have defined it here) is related to the concept of a toric frame, also introduced in \cite{BZ}.  The latter is a technical device used to make the general definition of a quantum cluster algebra.  For our purposes, it will suffice to think of $M$ simply as a rule determining the exchange monomials.

Next we turn to the definition of the quantized versions of matrix algebras and Grassmannians.  Let $q \in \mathbb{C}$ be non-zero and let $\mathbb{C}_q[M(2,n+3)]$ be the quantum matrix algebra, i.e.\ the $\mathbb{C}$-algebra generated by the set
$$
\{X_{ij} \mid 1 \leq i \leq 2,\;1 \leq j \leq n+3\},
$$
where the variables $X_{ij}$ are subject to the quantum matrix relations
\begin{align*}
X_{ki}X_{kj} & = qX_{kj}X_{ki} & X_{1i}X_{2i} & = qX_{2i}X_{1i} \\ X_{1j}X_{2i} & = X_{2i}X_{1j} &  X_{1i}X_{2j} - X_{2j}X_{1i} & = (q-q^{-1})X_{1j}X_{2i}
\end{align*}
for all $k = 1,2$ and $1 \leq i, j \leq n+3$ with $i < j$.

\begin{defn}
The {\em quantum coordinate ring $\mathbb{C}_q[\Gr(2,n+3)]$ of the Grassmannian} $\Gr(2,n+3)$ (or quantum Grassmannian for short) is the subalgebra of $\mathbb{C}_q[M(2,n+3)]$ generated by the quantum Pl\"ucker coordinates
$$
\Delta_q^{ij} =  X_{1i}X_{2j} - qX_{1j}X_{2i}
$$
for all $1 \leq i < j \leq n+3$.
\end{defn}

The quantum Pl\"ucker coordinates $\Delta_q^{ij}$ are subject to the {\em short quantum Pl\"ucker relations}
$$
\Delta_q^{ij}\Delta_q^{kl} = q^{-1}\Delta_q^{ik}\Delta_q^{jl} + q\Delta_q^{il}\Delta_q^{jk},
$$
see for example \cite{KLR}.

In \cite{GL}, the first author and Launois showed that $\mathbb{C}_q[\Gr(2,n+3)]$ has a quantum cluster algebra structure of type $A_n$.  There is a well-known correspondence between Pl\"ucker coordinates of the Grassmannian of planes in $n$-space and  diagonals (including the sides) of the $(n+3)$-gon. Furthermore, by Theorem 1.1 of \cite{LZ}, two Pl\"ucker coordinates $\Delta_q^{ij}$ and $\Delta_q^{kl}$ quasi-commute if and only if the corresponding arcs $(i,j)$ and $(k,l)$ in the $(n+3)$-gon do not cross.  So any triangulation of the $(n+3)$-gon together with its sides corresponds to a maximal set of quasi-commuting variables. In fact, it follows from \cite{GL} that such a triangulation gives rise to a cluster (and vice versa) just as in the classical case.

\subsection{Cluster categories and cluster structures}

A major step towards a better understanding of cluster algebras has been obtained by their categorification as first studied in \cite{BMRRT}. The basic idea is that certain subcategories in a triangulated Hom-finite Krull--Schmidt category take the role of clusters, with their indecomposable objects representing the cluster variables.  (See for example \cite{HJR} for an introduction to triangulated categories and their properties.)

\begin{defn}
Let $\mathcal{C}$ be a triangulated, Hom-finite Krull-Schmidt category with shift functor $\Sigma$. A subcategory $\mathcal{B}$ of $\mathcal{C}$ is called {\em weak cluster tilting} if $\mathcal{B} = (\Sigma^{-1}\mathcal{B})^{\perp} ={}^{\perp}(\Sigma \mathcal{B})$. A weak cluster tilting subcategory is called {\em cluster tilting} if in addition it is functorially finite. The collection of (isomorphism classes of) indecomposable objects in a cluster tilting subcategory is called a {\em cluster}.
\end{defn}

\begin{defn}[{\cite[Definition~5.1]{HJ}}]\label{cluster structure}
Let $\mathcal{C}$ be a triangulated Hom-finite Krull--Schmidt category with shift functor $\Sigma$ and let $\mathcal{T}$ be the collection of clusters in $\mathcal{C}$; that is, $\mathcal{T}$ consists of all indecomposable objects in $\mathcal{C}$ that lie in some cluster tilting subcategory of $\mathcal{C}$.  Then $\mathcal{T}$ imposes a {\em weak cluster structure} on $\mathcal{C}$ if the following hold:
\begin{itemize}
\item{for every cluster $C$ and every indecomposable object $c$ in $C$, there exists a unique indecomposable object $c^{*}$ of $\mathcal{C}$ such that $(C\setminus \{ c\}) \cup \{ c^{*} \}$ is again a cluster; and}
\item{there are distinguished triangles $c^{*} \to b \to c$ and $c \to b^{\prime} \to c^{*}$ in $\mathcal{C}$ such that the left-hand morphisms are $\add(C\setminus \{ c\})$-envelopes and the right-hand morphisms are $\add(C\setminus \{ c\})$-covers.}
\end{itemize}
A weak cluster structure is called a {\em cluster structure} if in addition the following hold:
\begin{itemize}
\item{if $C$ is a cluster, the quiver of $\add(C)$ has no loops or $2$-cycles; and}
\item{the quivers of $\add(C)$ and $\add((C\setminus \{ c\})\cup c^{*})$ are related by Fomin--Zelevinsky quiver mutation at $c$.}
\end{itemize}
\end{defn}

We note that this version of the definition of a (weak) cluster structure is better adapted to the situation where cluster tilting subcategories of $\mathcal{C}$ contain infinitely many isomorphism classes of indecomposable objects, as compared with the original definition of \cite{BIRS}.

Let $k$ be a field and view the polynomial algebra $k[X]$ in one variable over $k$ as a differential graded algebra with $X$ placed in homological degree $1$ and zero differentials. 

\begin{defn}\label{category-D}
Let $D$ be the derived category $D = D^f(k[X])$ of differential graded $k[X]$-modules with finite dimensional homology over $k$.
\end{defn}
This category has Auslander--Reiten quiver $\ZZ \vec{A}_{+\infty}$ where $\vec{A}_{+\infty}$ is the linearly-oriented quiver 
\begin{displaymath}
\xymatrix{\bullet \ar@{->}[r] & \bullet \ar@{->}[r] & \bullet \ar@{->}[r] & \bullet \ar@{->}[r] & \bullet \ar@{->}[r] & \ldots}
\end{displaymath}
with set of vertices $\{r \in \mathbb{Z}_{\geq 0}\}$ and arrows from $r$ to $r+1$ for all $r \geq 0$.  (We use the notation $\vec{\Gamma}$ to indicate that we are considering a quiver with underlying graph $\Gamma$ and say that $\vec{\Gamma}$ has type $\Gamma$.  We also note that we are using $A_{+\infty}$ as distinguished from the doubly-infinite graph $A_{\pm \infty}$ (sometimes denoted $A_{\infty}$ or $A^{\infty}_{\infty}$) having vertex set $\ZZ$ and edges between $r$ and $r+1$ for all $r$.)

Then the main result of \cite{HJ} is the existence of a cluster structure on $D$.

\begin{theorem}[{\cite[Theorem C]{HJ}}]
The clusters form a cluster structure on $D$. \qed
\end{theorem}

The inspiration for the present work is the observation in \cite{HJ} that there is a relationship between the cluster tilting subcategories of $D$ and triangulations of the $\infty$-gon.  Here, we view the $\infty$-gon as a line of integers

\begin{center}
\begin{tikzpicture}[scale =.75]
\tikzstyle{every node}=[font=\small]
\draw (-6.5,0) -- (7.5,0);
\foreach \x in {-6,...,7}
	\draw (\x, 0.1) -- (\x, -0.1) node [below] {\x};
	\draw[] (7.5,0) -- (8,0) node [below] {$\ldots$};
	\draw[] (-6.5,0) -- (-7,0) node [below] {$\ldots$};
\end{tikzpicture}
\end{center}
An arc in the $\infty$-gon is a pair of integers $(m,n)$, such that $m \leq n-2$. 

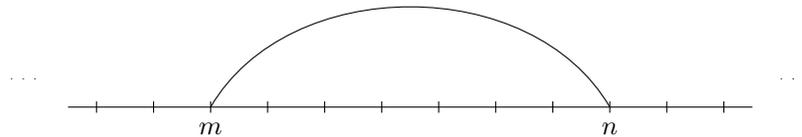
\begin{center}
\begin{tikzpicture}[scale =.75]
\tikzstyle{every node}=[font=\small]
\draw (-5.5,0) -- (6.5,0);
\foreach \x in {-5,...,-4}
	\draw (\x, 0.1) -- (\x, -0.1);
\draw (-3,0.1) -- (-3,-0.1) node[below]{$m$};
\foreach \x in {-2,...,3}
	\draw (\x, 0.1) -- (\x, -0.1);
\draw (4,0.1) -- (4,-0.1) node[below]{$n$};
\foreach \x in {5,...,6}
	\draw (\x, 0.1) -- (\x, -0.1);
\draw[loosely dotted] (-6.5,0.5) -- (-6,0.5);
\draw[loosely dotted] (7,0.5) -- (7.5,0.5);
\path (-3,0) edge [out= 60, in= 120] (4,0);
\end{tikzpicture}
\end{center}
For pairs of integers of the form $(n,n+1)$, we speak of {\em sides} of the $\infty$-gon. We say that two arcs $(i,j)$ and $(m,n)$ in the $\infty$-gon intersect if either $i<m<j<n$ or $m<i<n<j$.  This of course corresponds to the geometric picture of intersecting arcs:

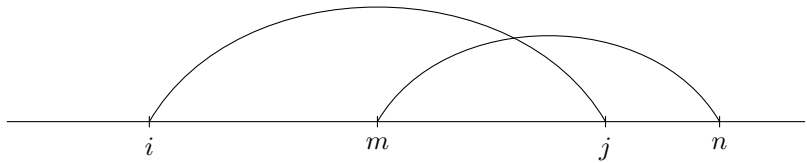
\begin{center}
\begin{tikzpicture}[scale =.75]
\tikzstyle{every node}=[font=\small]
\draw (-5.5,0) -- (8.5,0);
\draw (-3,0.1) -- (-3,-0.1) node[below]{$i$};
\draw (1,0.1) -- (1,-0.1) node[below]{$m$};
\draw (5,0.1) -- (5,-0.1) node[below]{$j$};
\draw (7,0.1) -- (7,-0.1) node[below]{$n$};

\path (-3,0) edge [out= 60, in= 120] (5,0);
\path (1,0) edge [out= 60, in= 120] (7,0);

\end{tikzpicture}
\end{center}

Now there is a natural bijection between the isomorphism classes of indecomposable objects of $D$ and arcs in the $\infty$-gon.  Following \cite[Remark~1.4]{HJ}, the Auslander--Reiten quiver of $D$, $\ZZ \vec{A}_{+\infty}$, may be given a coordinate system such that its vertices (corresponding to the isomorphism classes of indecomposable objects) are indexed by ordered pairs of integers $(m,n)$ with $m\leq n-2$; we refer the reader to \cite[page~281]{HJ} for a diagram.  Then each such pair naturally describes an arc in the $\infty$-gon and vice versa.

\begin{defn}
A {\em triangulation} of the $\infty$-gon is a maximal set of non-intersecting arcs. A triangulation is called {\em locally finite} if for every integer $n$ there are only finitely many arcs of the form $(m,n)$ or $(n,p)$. An integer $n$ in a non-locally finite triangulation is called a {\em right-fountain}, respectively a {\em left-fountain}, if there are infinitely many arcs of the form $(n,p)$, respectively of the form $(m,n)$. An integer $n$ is called a {\em fountain} of a triangulation if it is a right-fountain as well as a left-fountain in this triangulation. \end{defn}

Triangulations are important in this setting due to the following theorem.

\begin{theorem}[{\cite[Theorems 4.3 and 4.4]{HJ}}] \label{triangulations correspond to maximal 1-orthogonal subcategories}
A subcategory $\T$ of $D$ is weak cluster tilting if and only if the corresponding set of arcs $\mathfrak{t}$ is a triangulation of the $\infty$-gon. Furthermore, it is a cluster tilting subcategory if and only if the triangulation $\mathfrak{t}$ is locally finite or has a fountain. \qed
\end{theorem}

\section{Infinite rank cluster algebra structures from triangulations of the \texorpdfstring{$\infty$}{infinity}-gon}\label{Cluster algebra structures from triangulations}

Let $\vec{A}_n$ be any orientation of the Dynkin diagram $A_n$. In \cite{FZ2} it was noted that the combinatorics of cluster algebras associated to $\vec{A}_n$ can be expressed by the model of triangulations of the $(n+3)$-gon. The mutable cluster variables correspond to diagonals in $(n+3)$-gon, clusters correspond to triangulations and mutation has its geometric equivalent in the diagonal flip, which replaces a diagonal in a triangulation by the unique other diagonal such that the result is again a triangulation.  As noted above, this same combinatorial model encodes the behaviour of Pl\"{u}cker coordinates for the Grassmannian $\Gr(2,n+3)$, with a diagonal between vertices $i$ and $j$ corresponding to a minor $\Delta^{ij}$, and the short Pl\"{u}cker relations reflected in diagonal flips.  So these identifications allows us to see that the homogeneous coordinate ring $\mathbb{C}[\Gr(2,n+3)]$ of the Grassmannian $\Gr(2,n+3)$, obtained via the Pl\"ucker embedding, is a cluster algebra of type $A_{n}$.

To extend this theory to the infinite case and  make the step from the cluster category $D$ of \cite{HJ} (c.f.\ Definition \ref{category-D}) to cluster algebras, we define cluster algebras of infinite rank. 

\begin{defn}\label{cluster algebras of infinite rank}
Let $Q$ be a quiver without loops or $2$-cycles and with a countably infinite number of vertices labelled by all integers $i \in \mathbb{Z}$. Furthermore, for each vertex $i$ of $Q$ let the number of arrows incident with $i$ be finite. Consider the field of rational functions $\mathbb{Q}(\{ X_i\}_{i \in \mathbb{Z}})$ over a countably infinite set of variables. An {\em infinite initial seed} is the pair $(Q, \{X_i\}_{i \in \mathbb{Z}})$ consisting of the quiver $Q$ and the set of generating variables $\{X_i\}_{i \in \mathbb{Z}}$. By finite sequences of mutation at vertices of $Q$ and simultaneous mutation of the set $\{X_i\}_{i \in \mathbb{Z}}$ using the exchange relation (\ref{exchange relation}) as previously, we get a family of infinite seeds. The sets of variables in these seeds are called the {\em infinite clusters} and their elements are called the {\em cluster variables}. The {\em cluster algebra of infinite rank of type $Q$} is the subalgebra of $\mathbb{Q}(\{ X_i\}_{i \in \mathbb{Z}})$ generated by the cluster variables.
\end{defn}

Note that the requirement that every vertex of $Q$ is incident with only finitely many arrows ensures that the products in the exchange relations are finite and thus well-defined mutated variables are obtained.

We proceed in three steps.  First we examine the structure of the category $D$ in more detail and classify the exchange quivers associated to the resulting triangulations of the $\infty$-gon.  Next we examine the mutation equivalence classes of these and finally we use this information to obtain infinite rank cluster algebra structures.

\subsection{Classification of exchange quivers associated to triangulations}

To find cluster algebras of infinite rank that correspond to cluster structures on the category $D$ we begin by examining the triangulations of the $\infty$-gon, which we know to correspond to cluster tilting subcategories of $D$. On the categorical level, we want mutations of clusters to correspond to mutations of cluster tilting subcategories at (isomorphism classes of) indecomposables in the category $D$, which again correspond to diagonal flips in triangulations of the $\infty$-gon. From this point of view it is crucial that we only allow finite sequences of mutations in Definition \ref{cluster algebras of infinite rank}, as is illustrated by the following example.

\begin{ex}\label{ex:fin-many-mutations}
Let the triangulation $\mathfrak{t}$ consist of the standard fountain at $0$, i.e.\ the triangulation $\mathfrak{t}=\{(-n,0),(0,n)\}_{n \geq 2}$.

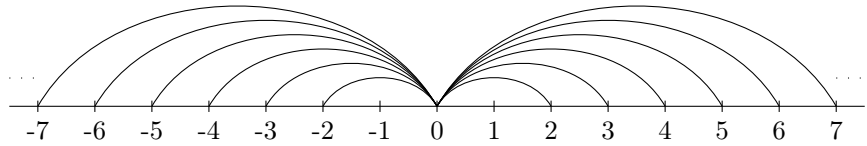
\begin{center}
\begin{tikzpicture}[scale = .75]
\tikzstyle{every node}=[font=\small]
\draw (-7.5,0) -- (7.5,0);
\foreach \x in {-7,...,7}
	\draw (\x, 0.1) -- (\x, -0.1) node [below] {\x};
\draw[loosely dotted] (-7.5,0.5) -- (-7,0.5);
\draw[loosely dotted] (7,0.5) -- (7.5,0.5);
\path (0,0) edge [out= 60, in= 120] (2,0);
\path (0,0) edge [out= 60, in= 120] (3,0);
\path (0,0) edge [out= 60, in= 120] (4,0);
\path (0,0) edge [out= 60, in= 120] (5,0);
\path (0,0) edge [out= 60, in= 120] (6,0);
\path (-5,0) edge [out= 60, in= 120] (0,0);
\path (-4,0) edge [out= 60, in= 120] (0,0);
\path (-6,0) edge [out= 60, in= 120] (0,0);
\path (-3,0) edge [out= 60, in= 120] (0,0);
\path (-2,0) edge [out= 60, in= 120] (0,0);
\path (0,0) edge [out= 60, in= 120] (7,0);
\path (-7,0) edge [out= 60, in= 120] (0,0);

\end{tikzpicture}
\end{center}
Let $\mu_k$ denote the diagonal flip of the arc $(0,k)$ and let $\mathfrak{t}_k$ be the triangulation 
\begin{align*}
\mathfrak{t}_k & = \mu_k \circ \mu_{k-1} \circ \cdots \circ \mu_3 \circ \mu_2 (\mathfrak{t})\\
& = \{(-n,0)\}_{n \in \mathbb{Z}_{\geq 2}} \cup \{(1,3),(1,4), \ldots, (1,k)\} \cup \{(0,k+m)\}_{m \in \mathbb{Z}_{\geq 1}}.
\end{align*}
Pictorially, $\mathfrak{t}_{k-1}$ is transformed into $\mathfrak{t}_{k}$ by $\mu_{k}$ as follows:
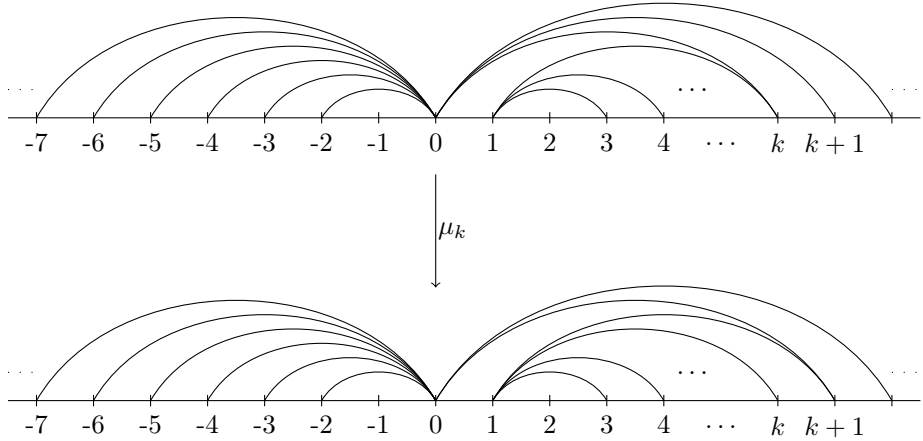
\begin{center}
\begin{tikzpicture}[scale = .75]
\tikzstyle{every node}=[font=\small]
\draw (-7.5,0) -- (8.5,0);
\foreach \x in {-7,...,4}
	\draw (\x, 0.1) -- (\x, -0.1) node [below] {\x};
	\node[below] (Punkte) at (5, -0.25) {$\ldots$};
	\draw (6, 0.1) -- (6, -0.1) node [below] {$k$};
	\draw (7, 0.1) -- (7, -0.1) node [below] {$k+1$};
	\draw (8, 0.1) -- (8, -0.1);
\draw[loosely dotted] (-7.5,0.5) -- (-7,0.5);
\draw[loosely dotted] (8,0.5) -- (8.5,0.5);
\path (1,0) edge [out= 60, in= 120] (3,0);
\path (1,0) edge [out= 60, in= 120] (4,0);
\node (punkte) at (4.5,0.5) {$\ldots$};
\path (1,0) edge [out= 60, in= 120] (6,0);
\path (-5,0) edge [out= 60, in= 120] (0,0);
\path (-4,0) edge [out= 60, in= 120] (0,0);
\path (-6,0) edge [out= 60, in= 120] (0,0);
\path (-3,0) edge [out= 60, in= 120] (0,0);
\path (-2,0) edge [out= 60, in= 120] (0,0);
\path (0,0) edge [out= 60, in= 120] (6,0);
\path (0,0) edge [out= 60, in= 120] (7,0);
\path (0,0) edge [out= 60, in= 120] (8,0);
\path (-7,0) edge [out= 60, in= 120] (0,0);

\draw[->] (0,-1) --  (0,-3); 
\node (mu) at (0.3,-2) {$\mu_k$};

\draw[yshift=-5cm] (-7.5,0) -- (8.5,0);
\foreach \x in {-7,...,4}
	\draw[yshift=-5cm] (\x, 0.1) -- (\x, -0.1) node [below] {\x};
	\node[yshift=-3.75cm,below] (Punkte) at (5, -0.25) {$\ldots$};
	\draw[yshift=-5cm] (6, 0.1) -- (6, -0.1) node [below] {$k$};
	\draw[yshift=-5cm] (7, 0.1) -- (7, -0.1) node [below] {$k+1$};
	\draw[yshift=-5cm] (8, 0.1) -- (8, -0.1);
\draw[yshift=-5cm,loosely dotted] (-7.5,0.5) -- (-7,0.5);
\draw[yshift=-5cm,loosely dotted] (8,0.5) -- (8.5,0.5);
\path[yshift=-5cm] (1,0) edge [out= 60, in= 120] (3,0);
\path[yshift=-5cm] (1,0) edge [out= 60, in= 120] (4,0);
\node[yshift=-3.75cm] (punkte) at (4.5,0.5) {$\ldots$};
\path[yshift=-5cm] (1,0) edge [out= 60, in= 120] (6,0);
\path[yshift=-5cm] (-5,0) edge [out= 60, in= 120] (0,0);
\path[yshift=-5cm] (-4,0) edge [out= 60, in= 120] (0,0);
\path[yshift=-5cm] (-6,0) edge [out= 60, in= 120] (0,0);
\path[yshift=-5cm] (-3,0) edge [out= 60, in= 120] (0,0);
\path[yshift=-5cm] (-2,0) edge [out= 60, in= 120] (0,0);
\path[yshift=-5cm] (0,0) edge [out= 60, in= 120] (7,0);
\path[yshift=-5cm] (1,0) edge [out= 60, in= 120] (7,0);
\path[yshift=-5cm] (0,0) edge [out= 60, in= 120] (8,0);
\path[yshift=-5cm] (-7,0) edge [out= 60, in= 120] (0,0);

\end{tikzpicture}
\end{center}
Allowing infinite sequences of diagonal flips would yield the split fountain $$\mathfrak{t}' = \{(-n,0),(1,n+1)\}_{n \geq 2}.$$

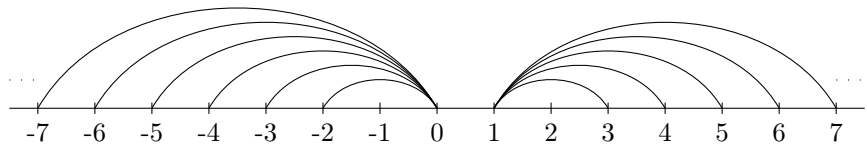
\begin{center}
\begin{tikzpicture}[scale = .75]
\tikzstyle{every node}=[font=\small]
\draw[yshift=-5cm] (-7.5,0) -- (7.5,0);
\foreach \x in {-7,...,7}
	\draw[yshift=-5cm] (\x, 0.1) -- (\x, -0.1) node [below] {\x};
\draw[yshift=-5cm,loosely dotted] (-7.5,0.5) -- (-7,0.5);
\draw[yshift=-5cm,loosely dotted] (7,0.5) -- (7.5,0.5);
\path[yshift=-5cm] (1,0) edge [out= 60, in= 120] (3,0);
\path[yshift=-5cm] (1,0) edge [out= 60, in= 120] (4,0);
\path[yshift=-5cm] (1,0) edge [out= 60, in= 120] (5,0);
\path[yshift=-5cm] (1,0) edge [out= 60, in= 120] (6,0);
\path[yshift=-5cm] (-5,0) edge [out= 60, in= 120] (0,0);
\path[yshift=-5cm] (-4,0) edge [out= 60, in= 120] (0,0);
\path[yshift=-5cm] (-6,0) edge [out= 60, in= 120] (0,0);
\path[yshift=-5cm] (-3,0) edge [out= 60, in= 120] (0,0);
\path[yshift=-5cm] (-2,0) edge [out= 60, in= 120] (0,0);
\path[yshift=-5cm] (1,0) edge [out= 60, in= 120] (7,0);
\path[yshift=-5cm] (-7,0) edge [out= 60, in= 120] (0,0);

\end{tikzpicture}
\end{center}
Repeating the same pattern of infinite mutations reflected on the left-fountain yields the following configuration of arcs:
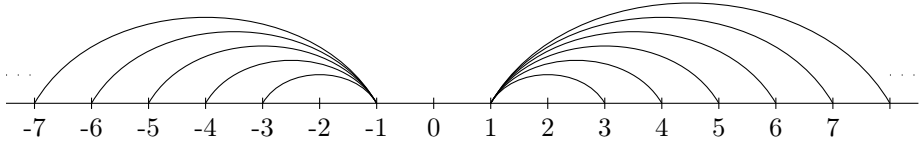
\begin{center}
\begin{tikzpicture}[scale = .75]
\tikzstyle{every node}=[font=\small]
\draw[yshift=-5cm] (-7.5,0) -- (8.5,0);
\foreach \x in {-7,...,7}
	\draw[yshift=-5cm] (\x, 0.1) -- (\x, -0.1) node [below] {\x};
	\draw[yshift=-5cm] (8, 0.1) -- (8, -0.1);
\draw[yshift=-5cm,loosely dotted] (-7.5,0.5) -- (-7,0.5);
\draw[yshift=-5cm,loosely dotted] (8,0.5) -- (8.5,0.5);
\path[yshift=-5cm] (1,0) edge [out= 60, in= 120] (3,0);
\path[yshift=-5cm] (1,0) edge [out= 60, in= 120] (4,0);
\path[yshift=-5cm] (1,0) edge [out= 60, in= 120] (5,0);
\path[yshift=-5cm] (1,0) edge [out= 60, in= 120] (6,0);
\path[yshift=-5cm] (-5,0) edge [out= 60, in= 120] (-1,0);
\path[yshift=-5cm] (-4,0) edge [out= 60, in= 120] (-1,0);
\path[yshift=-5cm] (-6,0) edge [out= 60, in= 120] (-1,0);
\path[yshift=-5cm] (-3,0) edge [out= 60, in= 120] (-1,0);
\path[yshift=-5cm] (1,0) edge [out= 60, in= 120] (7,0);
\path[yshift=-5cm] (1,0) edge [out= 60, in= 120] (8,0);
\path[yshift=-5cm] (-7,0) edge [out= 60, in= 120] (-1,0);

\end{tikzpicture}
\end{center}
which is not a triangulation of the $\infty$-gon, and hence does not correspond to a (weak) cluster tilting subcategory in $D$, by Theorem \ref{triangulations correspond to maximal 1-orthogonal subcategories}.
\end{ex}

To find the appropriate quiver $Q$ to encode the combinatorial data of the cluster algebra we are looking for, we generalize a well-known algorithm from the finite rank type $A_{n}$ case---for the construction of the exchange quiver for a cluster from a given triangulation of an $(n+3)$-gon---to triangulations of the $\infty$-gon.

%GENAUE REFERENZ FÜR EXCHANGE QUIVER?

\begin{defn} \label{exchange quiver}
Let $\mathfrak{t}$ be a triangulation of the $\infty$-gon. Then we construct the exchange quiver $Q_{\mathfrak{t}}$ by the following algorithm.
\begin{itemize}
\item[(1)]{For each side $(i,i+1)$ of the $\infty$-gon there is a frozen vertex in $Q_{\mathfrak{t}}$ which we label by $(i,i+1)$.}
\item[(2)]{For each arc $(i,j)$ (with $j\neq i+1$) in $\mathfrak{t}$ there is a mutable vertex in $Q_{\mathfrak{t}}$ which we label by $(i,j)$.}
\item[(3)]{We put a clockwise orientation on all triangles arising in the triangulation and draw arrows between the sides and diagonals accordingly, as illustrated below. Then there is an arrow from a vertex $(i,j)$ to a vertex $(k,l)$ in $Q_{\mathfrak{t}}$ whenever there is an equally-oriented arrow between the corresponding diagonals or sides and where we discard all arrows between frozen vertices.}
%DAS MIT DEN FROZEN VERTICES BESSER ERKLÄREN.
\end{itemize}

Triangulation:
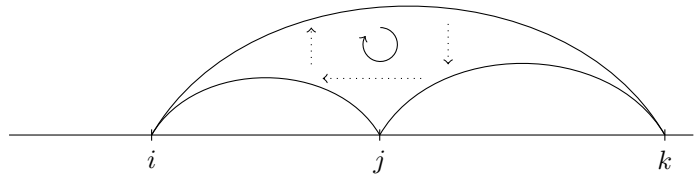
\begin{center}
\begin{tikzpicture}[scale =0.75]
\tikzstyle{every node}=[font=\small]
\draw (-5.5,0) -- (6.5,0);
\draw (-3,0.1) -- (-3,-0.1) node[below]{$i$};
\draw (1,0.1) -- (1,-0.1) node[below]{$j$};
\draw (6,0.1) -- (6,-0.1) node[below]{$k$};

\path (-3,0) edge [out= 60, in= 120] (1,0);
\path (1,0) edge [out= 60, in= 120] (6,0);
\path(-3,0) edge [out= 60, in= 120] (6,0);
\draw[<-] (0.75,1.75) arc (60+90:360+90:.3);
\draw[<-, dotted] (0,1) -- (1.8,1);
\draw[<-, dotted] (2.2,1.25) -- (2.2,2);
\draw[->, dotted] (-0.2,1.25) -- (-0.2,1.9);
\end{tikzpicture}
\end{center}

Exchange quiver:
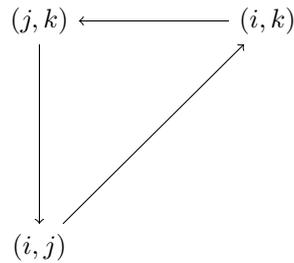
\begin{center}
\begin{tikzpicture}
\tikzstyle{every node}=[font=\small]
\node (1) at (0,0) {$(i,j)$};
\node (2) at (0,3) {$(j,k)$};
\node (3) at (3,3) {$(i,k)$};

\draw[<-] (1) -- (2);
\draw[<-] (2) -- (3);
\draw[<-] (3) -- (1);
\end{tikzpicture}
\end{center}

\end{defn}

\begin{ex}
Consider the ``standard leapfrog'' triangulation
$$
\mathfrak{t} = \{(-n,n)\}_{n \in \mathbb{Z}_{\geq 1}} \cup \{(-m,m+1)\}_{m \in \mathbb{Z}_{\geq 1}}.
$$
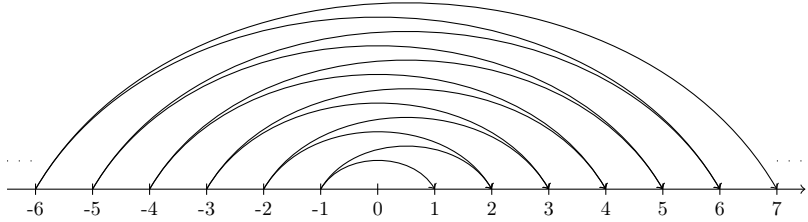
\begin{center}
\begin{tikzpicture}[scale = .75, transform shape]
\tikzstyle{every node}=[font=\small]
\draw[->] (-6.5,0) -- (7.5,0);
\foreach \x in {-6,...,7}
	\draw (\x, 0.1) -- (\x, -0.1) node [below] {\x};
\draw[loosely dotted] (-6.5,0.5) -- (-6,0.5);
\draw[loosely dotted] (7,0.5) -- (7.5,0.5);
\path[->] (-1,0) edge [out= 60, in= 120] (1,0);
\path[->] (-1,0) edge [out= 60, in= 120] (2,0);
\path[->] (-2,0) edge [out= 60, in= 120] (2,0);
\path[->] (-2,0) edge [out= 60, in= 120] (3,0);
\path[->] (-3,0) edge [out= 60, in= 120] (3,0);
\path[->] (-3,0) edge [out= 60, in= 120] (4,0);
\path[->] (-4,0) edge [out= 60, in= 120] (4,0);
\path[->] (-4,0) edge [out= 60, in= 120] (5,0);
\path[->] (-5,0) edge [out= 60, in= 120] (5,0);
\path[->] (-5,0) edge [out= 60, in= 120] (6,0);
\path[->] (-6,0) edge [out= 60, in= 120] (6,0);
\path[->] (-6,0) edge [out= 60, in= 120] (7,0);
\end{tikzpicture}
\end{center}
The corresponding exchange quiver $Q_{\mathfrak{t}}$ has as principal part the quiver with underlying graph $A_{+\infty}$, with a sink-source orientation; see Figure \ref{fig:standard leapfrog} on page \pageref{fig:standard leapfrog}. The frozen vertices in this picture are marked by rectangles as usual.

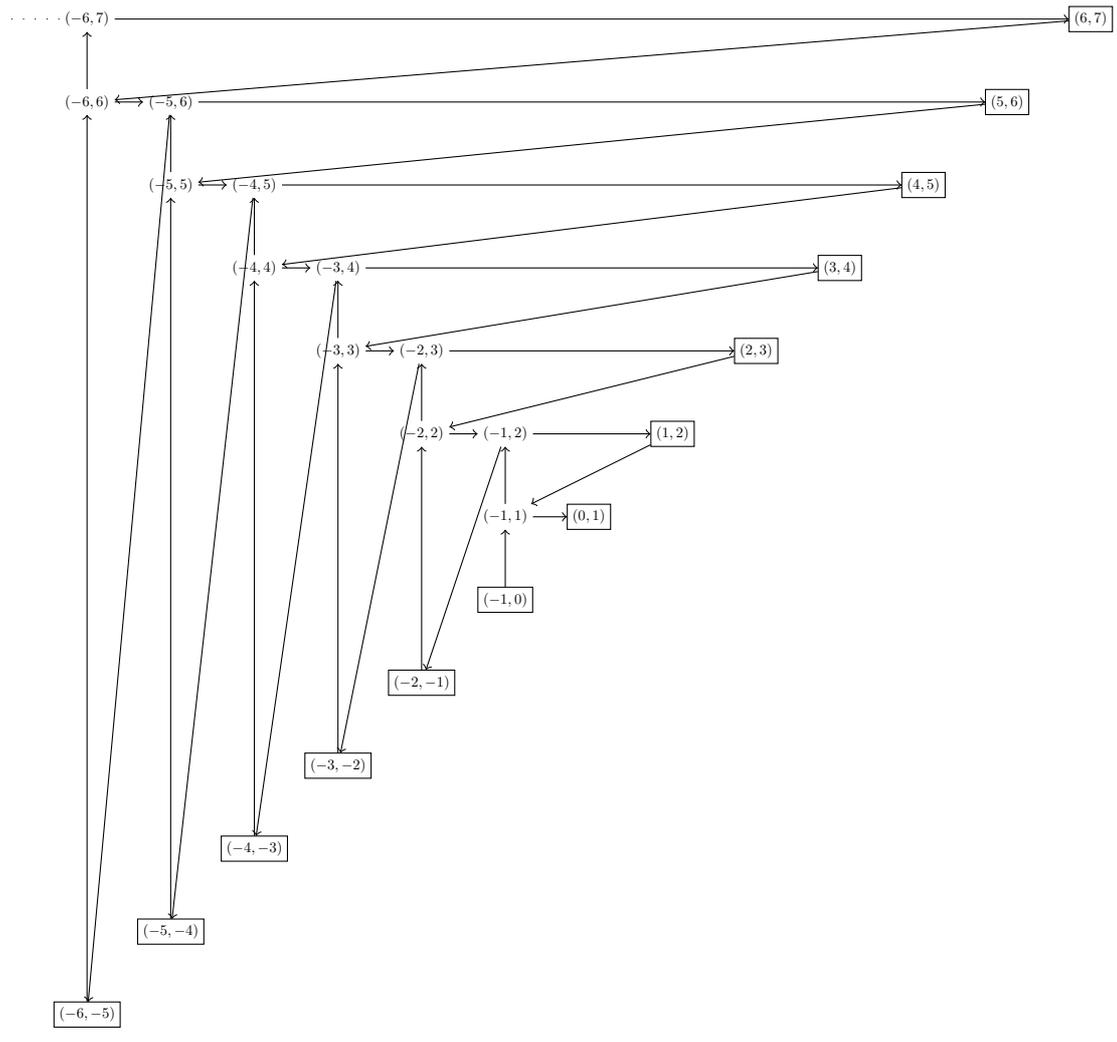
\begin{figure}
\begin{center}
\begin{tikzpicture}[scale=.55, transform shape]
\tikzstyle{every node}=[font=\small]
\node (-11) at (-2,2) {$(-1,1)$};
\node (-12) at (-2,4) {$(-1,2)$};
\node (-22) at (-4,4) {$(-2,2)$};
\node (-23) at (-4,6) {$(-2,3)$};
\node (-33) at (-6,6) {$(-3,3)$};
\node (-34) at (-6,8) {$(-3,4)$};
\node (-44) at (-8,8) {$(-4,4)$};
\node (-45) at (-8,10) {$(-4,5)$};
\node (-55) at (-10,10) {$(-5,5)$};
\node (-56) at (-10,12) {$(-5,6)$};
\node (-66) at (-12,12) {$(-6,6)$};
\node (-67) at (-12,14) {$(-6,7)$};

\node[draw, shape=rectangle] (01) at (0,2) {$(0,1)$};
\node[draw, shape=rectangle] (-10) at (-2,0) {$(-1,0)$};
\node[draw, shape=rectangle] (12) at (2,4) {$(1,2)$};
\node[draw, shape=rectangle] (23) at (4,6) {$(2,3)$};
\node[draw, shape=rectangle] (-2-1) at (-4,-2) {$(-2,-1)$};
\node[draw, shape=rectangle] (34) at (6,8) {$(3,4)$};
\node[draw, shape=rectangle] (-3-2) at (-6,-4) {$(-3,-2)$};
\node[draw, shape=rectangle] (45) at (8,10) {$(4,5)$};
\node[draw, shape=rectangle] (-4-3) at (-8,-6) {$(-4,-3)$};
\node[draw, shape=rectangle] (56) at (10,12) {$(5,6)$};
\node[draw, shape=rectangle] (-5-4) at (-10,-8) {$(-5,-4)$};
\node[draw, shape=rectangle] (67) at (12,14) {$(6,7)$};
\node[draw, shape=rectangle] (-6-5) at (-12,-10) {$(-6,-5)$};

\draw[->](-11) -- (01);
\draw[->](-10) -- (-11);
\draw[->](-11) -- (-12);
\draw[->](-12) -- (12);
\draw[->](12) -- (-11);
\draw[->](-2-1) -- (-22);
\draw[->](-22) -- (-12);
\draw[->](-12) -- (-2-1);
\draw[->](-22) -- (-23);
\draw[->](-23) -- (23);
\draw[->](23) -- (-22);
\draw[->](-33) -- (-23);
\draw[->](-23) -- (-3-2);
\draw[->](-3-2) -- (-33);
\draw[->](-33) -- (-34);
\draw[->](-34) -- (34);
\draw[->](34) -- (-33);
\draw[->](-44) -- (-34);
\draw[->](-4-3) -- (-44);
\draw[->](-34) -- (-4-3);
\draw[->](-44) -- (-45);
\draw[->](-45) -- (45);
\draw[->](45) -- (-44);
\draw[->](-55) -- (-45);
\draw[->](-45) -- (-5-4);
\draw[->](-5-4) -- (-55);
\draw[->](-55) -- (-56);
\draw[->](-56) -- (56);
\draw[->](56) -- (-55);
\draw[->](-66) -- (-56);
\draw[->](-56) -- (-6-5);
\draw[->](-6-5) -- (-66);
\draw[->](-66) -- (-67);
\draw[->](-67) -- (67);
\draw[->](67) -- (-66);
\draw[loosely dotted](-67) -- (-14,14);
\end{tikzpicture}
\end{center}
\caption{Quiver corresponding to the standard leapfrog triangulation of the $\infty$-gon.}
\label{fig:standard leapfrog}
\end{figure}
On the other hand, if we consider a triangulation with a fountain at $0$, which locally looks like this
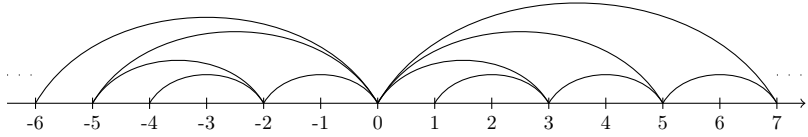
\begin{center}
\begin{tikzpicture}[scale = .75, transform shape]
\tikzstyle{every node}=[font=\small]
\draw[->] (-6.5,0) -- (7.5,0);
\foreach \x in {-6,...,7}
	\draw (\x, 0.1) -- (\x, -0.1) node [below] {\x};
\draw[loosely dotted] (-6.5,0.5) -- (-6,0.5);
\draw[loosely dotted] (7,0.5) -- (7.5,0.5);
\path (0,0) edge [out= 60, in= 120] (5,0);
\path (0,0) edge [out= 60, in= 120] (3,0);
\path (1,0) edge [out= 60, in= 120] (3,0);
\path (3,0) edge [out= 60, in= 120] (5,0);
\path (-5,0) edge [out= 60, in= 120] (0,0);
\path (-4,0) edge [out= 60, in= 120] (-2,0);
\path (-2,0) edge [out= 60, in= 120] (0,0);
\path (-5,0) edge [out= 60, in= 120] (-2,0);
\path (0,0) edge [out= 60, in= 120] (7,0);
\path (5,0) edge [out= 60, in= 120] (7,0);
\path (-6,0) edge [out= 60, in= 120] (0,0);
\end{tikzpicture}
\end{center}
we get an altogether different picture, see Figure \ref{fig:fountain} on page \pageref{fig:fountain}. In particular, the exchange quiver for this triangulation is disconnected.
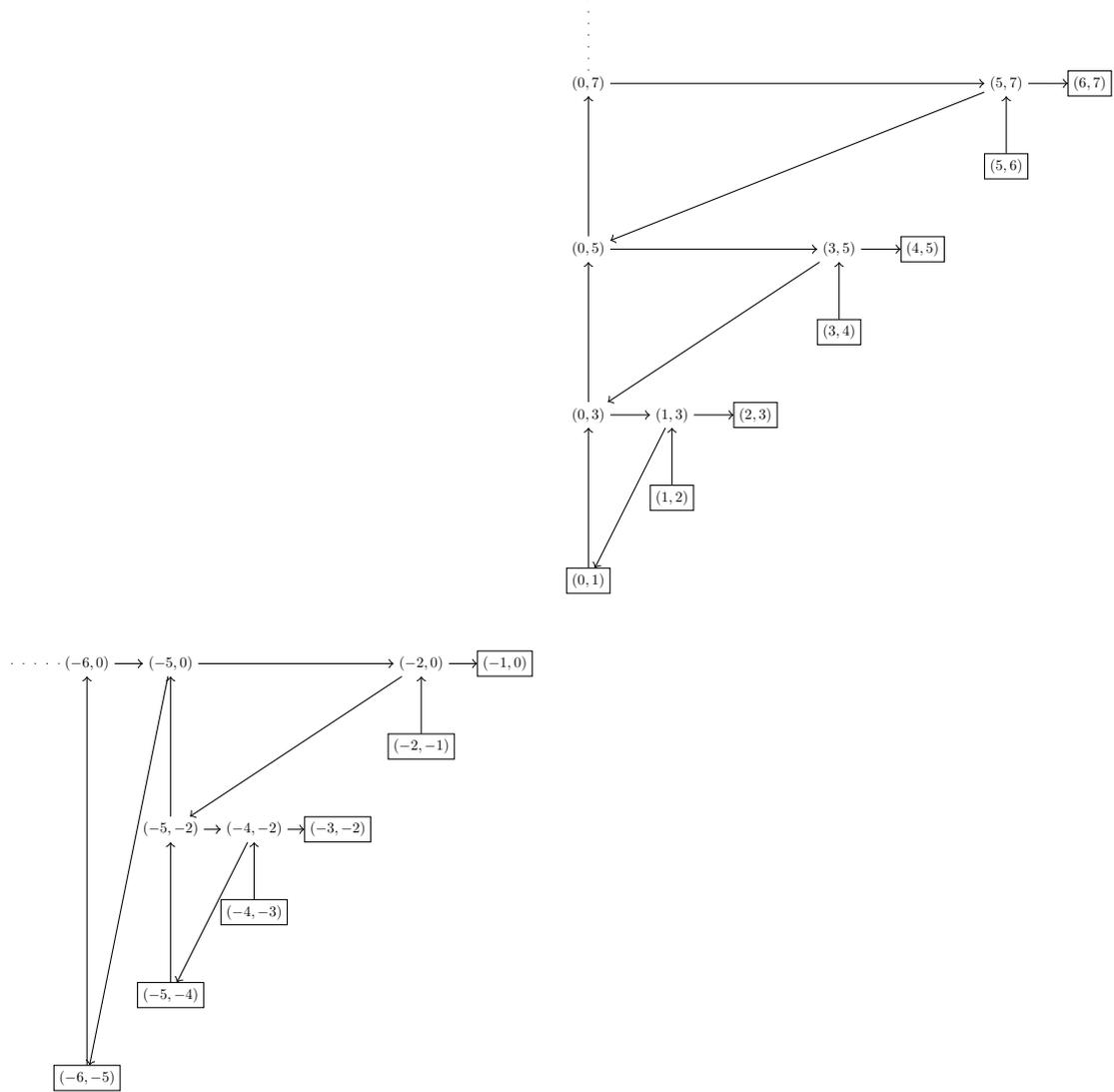
\begin{figure}
\begin{center}
\begin{tikzpicture}[scale=.55, transform shape]
\tikzstyle{every node}=[font=\small]
\node[draw, shape=rectangle] (01) at (0,2) {$(0,1)$};
\node[draw, shape=rectangle] (12) at (2,4) {$(1,2)$};
\node[draw, shape=rectangle] (23) at (4,6){$(2,3)$};
\node[draw, shape=rectangle] (34) at (6,8)  {$(3,4)$};
\node[draw, shape=rectangle] (45) at (8,10) {$(4,5)$};
%\node[draw, shape=rectangle] (56) at (19,12) {$(5,6)$};
\node[draw, shape=rectangle] (56) at (10,12) {$(5,6)$};
\node[draw, shape=rectangle] (67) at (12,14) {$(6,7)$};
\node (03) at (0,6) {$(0,3)$};
\node (05) at (0,10) {$(0,5)$};
\node (07) at (0,14){$(0,7)$};
\node (13) at (2,6) {$(1,3)$};
\node (35) at (6,10){$(3,5)$};
\node (57) at (10,14){$(5,7)$};

\draw[->] (01) -- (03);
\draw[->] (03) -- (05);
\draw[->] (05)-- (35);
\draw[->] (35)-- (45);
\draw[->] (34)-- (35);
\draw[->] (03)-- (13);
\draw[->] (13)-- (23);
\draw[->] (12)-- (13);
\draw[->] (13)-- (01);
\draw[->] (35)-- (03);
\draw[->] (05)-- (07);
\draw[->] (07)-- (57);
\draw[->] (57)-- (05);
\draw[->] (56)-- (57);
\draw[->] (57)-- (67);
\draw[loosely dotted] (07) -- (0,16);

\node (-60) at (-12,0) {$(-6,0)$};
\node (-50) at (-10,0){$(-5,0)$};
\node (-20) at (-4,0){$(-2,0)$};
\node (-5-2) at (-10,-4) {$(-5,-2)$};
\node (-4-2) at (-8,-4) {$(-4,-2)$};
\node[draw, shape=rectangle] (-10) at (-2,0){$(-1,0)$};
\node[draw, shape=rectangle] (-2-1) at (-4,-2){$(-2,-1)$};
\node[draw, shape=rectangle] (-3-2) at (-6,-4){$(-3,-2)$};
\node[draw, shape=rectangle] (-4-3) at (-8,-6){$(-4,-3)$};
\node[draw, shape=rectangle] (-5-4) at (-10,-8){$(-5,-4)$};
\node[draw, shape=rectangle] (-6-5) at (-12,-10){$(-6,-5)$};

\draw[->] (-20)-- (-10);
\draw[->] (-50)-- (-20);
\draw[->] (-5-2)-- (-50);
\draw[->] (-5-4)-- (-5-2);
\draw[->] (-20)-- (-5-2);
\draw[->] (-4-3)-- (-4-2);
\draw[->] (-4-2)-- (-3-2);
\draw[->] (-2-1)-- (-20);
\draw[->] (-5-2)-- (-4-2);
\draw[->] (-4-2)-- (-5-4);
\draw[->] (-60)-- (-50);
\draw[->] (-50)-- (-6-5);
\draw[->] (-6-5)-- (-60);
\draw[loosely dotted] (-60)-- (-14,0);
\end{tikzpicture}
\end{center}
\caption{Quiver for a fountain triangulation of the $\infty$-gon.}
\label{fig:fountain}
\end{figure}
\end{ex}

In general, the different types of triangulations of the $\infty$-gon yield different types of exchange quivers. We will need to consider separately each of the three different types arising from the following classification.

\begin{lm}[{\cite[Lemma~3.3]{HJ12}}]
Let $\mathfrak{t}$ be a triangulation of the $\infty$-gon. Then either $\mathfrak{t}$
\begin{itemize}
\item[(i)]{is locally finite,}
\item[(ii)]{has a fountain or}
\item[(iii)]{has a split fountain}. \qed
\end{itemize}
\end{lm}

%For the convenience of the reader, we give an independent proof.

%\begin{proof}
%Let $\mathfrak{t}$ be a triangulation which is not locally finite. We claim that it has exactly one left and one right-fountain.Indeed, it can have at most one right-, respectively left-, fountain. Assume by contradiction that both $r$ and $r'$ are right-fountains in $\mathfrak{t}$ with $r'<r$. Then as $r'$ is a right-fountain, there is an arc $(r',k')$ in $\mathfrak{t}$ with $k' > r$ and as $r$ is a right-fountain as well, there is a $k > k'$ such that $(r,k)$ is in $\mathfrak{t}$. But then the arcs $(r,k)$ and $(r',k')$ intersect, which contradicts the fact, that $\mathfrak{t}$ is a triangulation. By reflection, we see that $\mathfrak{t}$ can have at most one left-fountain.

%As $\mathfrak{t}$ is not locally finite, there is a $k \in \mathbb{Z}$ such that there are infinitely many arcs incident with $k$ in $\mathfrak{t}$. Suppose without loss of generality, that $k$ is a right-fountain. The case where we have a left-fountain follows from reflection. Assume by contradiction that there is no left-fountain in $\mathfrak{t}$. Then there are only finitely many arcs of the form $(j,k)$ in $\mathfrak{t}$. Take the minimal $j_0$ such that $(j_0,k)$ is in $k$ (possibly this could be the side $(k-1,k)$). Then take again the minimal $i_0$ such that $(i_0,j_0)$ is in $\mathfrak{t}$ (again this could be a side). Then the arc $(i_0,k)$ does not interesect any of the arcs in the triangulation $\mathfrak{t}$, hence has to lie in $\mathfrak{t}$ itself which contradicts the minimality of $j_0$.
%\end{proof}

To examine these different types, we use the following notions to talk about relationships between arcs within triangulations. 

\begin{defn}
Let $\mathfrak{t}$ be a triangulation of the $\infty$-gon. We say that an arc $(x,y) \in \mathfrak{t}$ {\em passes over} an arc or a side $(i,j)\in \mathfrak{t} \cup \{\text{sides of the $\infty$-gon}\}$ in the triangulation $\mathfrak{t}$ if $(x,y) \neq (i,j)$ and $x \leq i < j \leq y$. We say that it is a {\em minimal arc passing over} $(x,y)$ in the triangulation $\mathfrak{t}$  if every other arc $(x',y') \in \mathfrak{t}$ passing over $(i,j)$ also passes over $(x,y)$.

\end{defn}

The following technical result on the existence, uniqueness and form of minimal arcs is used repeatedly in the sequel.

\pagebreak

\begin{lm} \label{minimal arc}
Let $\mathfrak{t}$ be a triangulation of the $\infty$-gon that is locally finite or has a fountain. For any arc or side $(i,j) \in \mathfrak{t} \cup \{\text{sides of the $\infty$-gon}\}$ there exists a unique minimal arc passing over $(i,j)$ in $\mathfrak{t}$.  Furthermore we either have $x = i$ and we say that $(x,y)$ passes over $(i,j)$ {\em to the right}, or we have $y = j$ and we say that $(x,y)$ passes over $(i,j)$ {\em to the left}.
\end{lm}

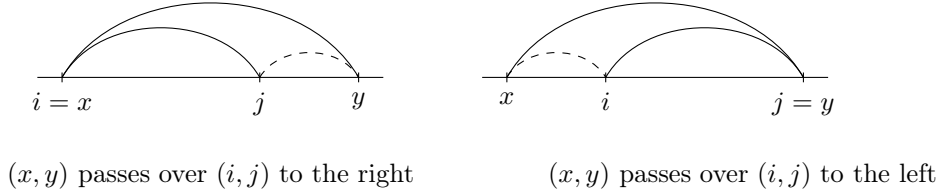
\begin{center}
\begin{tikzpicture}[scale =.65]
\tikzstyle{every node}=[font=\small]
\draw[xshift = 11cm] (-5.5,0) -- (1.5,0);
\draw[xshift = 11cm] (-3,0.1) -- (-3,-0.1) node[below]{$i$};
\draw[xshift = 11cm] (1,0.1) -- (1,-0.1) node[below]{$j = y$};
\draw[xshift = 11cm] (-5,0.1) -- (-5,-0.1) node[below]{$x$};

\path[xshift = 11cm] (-3,0) edge [out= 60, in= 120] (1,0);
\path[xshift = 11cm] (-5,0) edge [out= 60, in= 120] (1,0);
\path[xshift = 11cm, dashed] (-5,0) edge [out= 60, in= 120] (-3,0);
\node (text) at (0,-2) {$(x,y)$ passes over $(i,j)$ to the right};

\draw (-3.5,0) -- (3.5,0);
\draw (-3,0.1) -- (-3,-0.1) node[below]{$i=x$};
\draw (1,0.1) -- (1,-0.1) node[below]{$j$};
\draw (3,0.1) -- (3,-0.1) node[below]{$y$};

\path (-3,0) edge [out= 60, in= 120] (1,0);
\path (-3,0) edge [out= 60, in= 120] (3,0);
\path[dashed] (1,0) edge [out= 60, in= 120] (3,0);

\node[xshift = 7cm] (text) at (0,-2) {$(x,y)$ passes over $(i,j)$ to the left};

\end{tikzpicture}
\end{center}

\begin{proof}

Let $\mathfrak{t}$ be a triangulation of the $\infty$-gon and $(i,j)$ any arc or side in $\mathfrak{t} \cup \{\text{sides of the $\infty$-gon}\}$. We show that there is always an arc in $\mathfrak{t}$ passing over $(i,j)$.

First, consider the case where $\mathfrak{t}$ has a fountain at $k$ with $k \leq i < j$. As there are infinitely many arcs of the form $(k,l)$ with $l \in \mathbb{Z}$ there exists a $m \in \mathbb{Z}$ with $k \leq i < j < m$ and hence $(k,m)$ passes over $(i,j)$. The case $i < j \leq k$ follows from reflection. 

On the other hand, consider the case where $\mathfrak{t}$ is a locally finite triangulation. Assume for a contradiction that there is no arc passing over $(i,j)$. As there are infinitely many arcs in $\mathfrak{t}$ and by assumption there is no arc passing over $(i,j)$, there must be arcs to the right and left of it, that is, arcs of the form $(h,i')$ with $i' \leq i$, respectively $(j',k)$ with $j \leq j'$.  Let $i_0 \in \mathbb{Z}$ with $i_0 \leq i$ be maximal with the property that there exists an arc $(h,i_0)$ in $\mathfrak{t}$ and let $h_0 \in \mathbb{Z}$ be the minimal such $h$.  Similarly let  $j_0 \in \mathbb{Z}$ with $j_0 \geq i$ be minimal with the property that there exists an arc $(j_0,k)$ in $\mathfrak{t}$ and let $k_0 \in \mathbb{Z}$ be the maximal such $k$. But then the arc $(h_0,k_0)$ does not intersect any arcs in $\mathfrak{t}$, hence it has to be in $\mathfrak{t}$ and thus is an arc passing over $(i,j)$ in $\mathfrak{t}$, which contradicts the assumption.  This is illustrated by the following diagram:

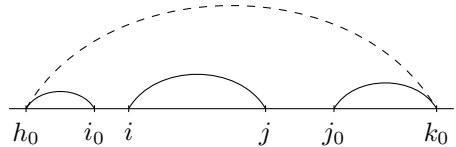
\begin{center}
\begin{tikzpicture}[scale =.45]
\tikzstyle{every node}=[font=\small]
\draw (-6.5,0) -- (6.5,0);
\draw (-3,0.1) -- (-3,-0.1) node[below]{$i$};
\draw (1,0.1) -- (1,-0.1) node[below]{$j$};
\draw (-6,0.1) -- (-6,-0.1) node[below]{$h_0$};
\draw (6,0.1) -- (6,-0.1) node[below]{$k_0$};
\draw (-4,0.1) -- (-4,-0.1) node[below]{$i_0$};
\draw (3,0.1) -- (3,-0.1) node[below]{$j_0$};

\path[dashed] (-6,0) edge [out= 60, in= 120] (6,0);
\path (-3,0) edge [out= 60, in= 120] (1,0);
\path (3,0) edge [out= 60, in= 120] (6,0);
\path (-6,0) edge [out= 60, in= 120] (-4,0);
%\path (3,0) edge [out= 60, in= 120] (5,0);
\end{tikzpicture}
\end{center}

So if $\mathfrak{t}$ is a locally finite triangulation or a triangulation with a fountain, there are always arcs passing over any arc $(i,j)$ in $\mathfrak{t}$. As the set of arcs passing over $(i,j)$ is bounded below (in the sense that the distance between the end-points  of any such arc cannot be smaller than $|i-j|$) there exists a minimal arc passing over $(i,j)$. 

Let $(x,y)$ be a minimal arc passing over $(i,j)$. Assume for a contradiction that $x < i < j < y$, then by minimality of $(x,y)$ the arc $(i,y)$ does not intersect any of the arcs in the triangulation. Hence it has to be in the triangulation itself, which contradicts the minimality of $(x,y)$:

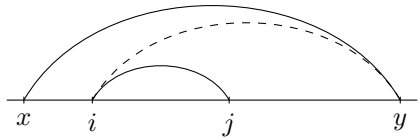
\begin{center}
\begin{tikzpicture}[scale =.45]
\tikzstyle{every node}=[font=\small]
\draw (-5.5,0) -- (6.5,0);
\draw (-3,0.1) -- (-3,-0.1) node[below]{$i$};
\draw (1,0.1) -- (1,-0.1) node[below]{$j$};
\draw (-5,0.1) -- (-5,-0.1) node[below]{$x$};
\draw (6,0.1) -- (6,-0.1) node[below]{$y$};

\path (-5,0) edge [out= 60, in= 120] (6,0);
\path (-3,0) edge [out= 60, in= 120] (1,0);
\path[dashed] (-3,0) edge [out= 60, in= 120] (6,0);
\end{tikzpicture}
\end{center}
Hence either we have $x = i$ or $y = j$. From this, we directly see the uniqueness of the minimal arc passing over $(i,j)$.
\end{proof}

\begin{rmk}
Note that in a triangulation with a split fountain, Lemma \ref{minimal arc} does not hold. Let $\mathfrak{t}$ be such a triangulation and let $l$ be a left-fountain and $r$ a right-fountain with $l \neq r$ in $\mathfrak{t}$. Then the arc $(l,r)$ which lies in $\mathfrak{t}$ has no arc passing over it. However, by the same arguments as in the proof of Lemma \ref{minimal arc} every arc in $\mathfrak{t}$ which is not $(l,r)$ does have a minimal arc passing over it. So, with careful treatment of the arc $(l,r)$ we can still use the concepts of minimal arcs passing over each other in the split fountain case.

We note that in a triangulation with a split fountain as above, the arc $(l,r)$ is not mutable.  That is, there is no possible diagonal flip of this arc.  Consequently, in a triangulation containing a split fountain, we may regard the arc $(l,r)$ as frozen.
\end{rmk}

Next we study connected components of the exchange quivers arising from triangulations.

\begin{lm}\label{vertices connected}
Let $\mathfrak{t}$ be a triangulation of the $\infty$-gon and $Q_{\mathfrak{t}}$ the associated exchange quiver. Two vertices $(i,j)$ and $(m,n)$ in the quiver $Q_{\mathfrak{t}}$ are in the same connected component of $Q_{\mathfrak{t}}$ if and only if there is an arc in $\mathfrak{t}$ passing simultaneously over $(i,j)$ and $(m,n)$.
\end{lm}

\begin{proof}

Suppose first there is an arc $(x,y)$ passing over both $(i,j)$ and $(m,n)$. Then we can view $(i,j)$ and $(m,n)$ locally as diagonals in the $(x-y+1)$-gon with vertices $\{x, x+1, \ldots, y-1,y\}$. From the finite case we know that the exchange graph for this finite-gon is mutation equivalent to a finite $A$-type quiver (c.f.\;\cite{FZ2}), in particular it is connected. The algorithm for the exchange quiver $Q_{\mathfrak{t}}$ yields this exchange quiver as a subquiver containing both $(i,j)$ and $(m,n)$. Hence they both lie in the same connected component of $Q_{\mathfrak{t}}$.

On the other hand, assume that there is no arc passing over both $(i,j)$ and $(m,n)$. As both arcs cannot pass over each other by assumption, they have to lie next to each other, so let $i < j \leq m < n$. Suppose as a contradiction that the vertices corresponding $(i,j)$ and $(m,n)$ in the exchange quiver are connected. Then by construction of $Q_{\mathfrak{t}}$there is a sequence of arcs
$$\{(a_k,b_k)\}_{k=0}^l,$$ such that $$(i,j) = (a_0,b_0), \;(m,n) = (a_l,b_l)$$ and for every $a \leq k \leq l$, $(a_k,b_k)$  and $ (a_{k-1},b_{k-1})$ are the sides of a common triangle, namely $$(a_{k-1},a_k, b_k)\text{ or }(a_k,b_k,b_{k-1}).$$
But then one of the arcs $(a_k,b_k)$ has to pass both over $(i,j)$ and $(m,n)$ which contradicts the assumption.
\end{proof}

In particular, Lemma \ref{vertices connected} allows us to obtain the number of connected components for the respective triangulations, as follows.

\begin{prop}\label{number of components} {\ }
\begin{itemize}
\item[i)]{The exchange quiver of a locally finite triangulation is connected.}
\item[ii)]{The exchange quiver of a triangulation with a fountain consists of two infinite connected components.}
\item[iii)]{The exchange quiver of a triangulation with a split fountain consists of three connected components, of which two are infinite and the third is finite.}
\end{itemize}
\end{prop}

\pagebreak

\begin{proof} {\ }
\begin{itemize}
\item[i)]{
Take any two arcs (or sides) $(i,j)$ and $(m,n)$ in a locally finite triangulation $\mathfrak{t}$. If either $(i,j)$ passes over $(m,n)$ or vice versa, then the corresponding vertices in the quiver are connected, by Lemma \ref{vertices connected}. Otherwise, let $i < j \leq m < n$. Then there is a maximal arc $(x,y)$ in $\mathfrak{t}$ with $j \leq x \leq m$ and $n \leq y$. Equally, there is a maximal arc $(u,v)$ in $\mathfrak{t}$ with $u \leq i$ and $j \leq v \leq m$. The arc $(u,y)$ passes over both $(i,j)$ and $(m,n)$ and since it does not intersect any of the arcs in $\mathfrak{t}$, it has to be in the triangulation $\mathfrak{t}$.

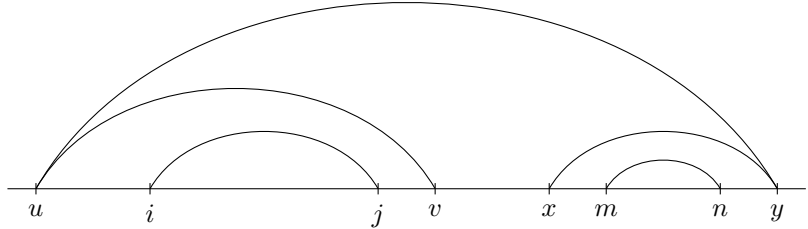
\begin{center}
\begin{tikzpicture}[scale =.75]
\tikzstyle{every node}=[font=\small]
\draw (-5.5,0) -- (8.5,0);
\draw (-3,0.1) -- (-3,-0.1) node[below]{$i$};
\draw (1,0.1) -- (1,-0.1) node[below]{$j$};
\draw (5,0.1) -- (5,-0.1) node[below]{$m$};
\draw (7,0.1) -- (7,-0.1) node[below]{$n$};
\draw (8,0.1) -- (8,-0.1) node[below]{$y$};
\draw (-5,0.1) -- (-5,-0.1) node[below]{$u$};
\draw (2,0.1) -- (2,-0.1) node[below]{$v$};
\draw (4,0.1) -- (4,-0.1) node[below]{$x$};

\path (-5,0) edge [out= 60, in= 120] (8,0);
\path (5,0) edge [out= 60, in= 120] (7,0);
\path (-3,0) edge [out= 60, in= 120] (1,0);
\path (4,0) edge [out= 60, in= 120] (8,0);
\path (-5,0) edge [out= 60, in= 120] (2,0);
\end{tikzpicture}
\end{center}
By Lemma \ref{vertices connected}, the vertices $(i,j)$ and $(m,n)$ lie in the same connected component of $Q_{\mathfrak{t}}$ and as they were chosen arbitrarily, the exchange quiver is connected.
}
\item[ii)/iii)]{
Let $\mathfrak{t}$ be a triangulation with a right-fountain $k$, let $(i,j)$ and $(m,n)$ be two arcs in $\mathfrak{t}$ with $i,j,m,n \geq k$. As $k$ is a right-fountain, there is an integer $l$ with $l \geq j$, $l \geq n$ such that $(k,l)$ is an arc in $\mathfrak{t}$ and $(k,l)$ passes over both $(i,j)$ and $(m,n)$.
Hence all vertices corresponding to arcs on the right-hand side of a fountain are connected. By reflection, all vertices corresponding to arcs on the left-hand side of a left-fountain are connected. This shows that the quiver corresponding to a triangulation with a fountain, respectively a split fountain, has at most $2$, respectively $3$, connected components.

Let $k$ be a right- or left-fountain in a triangulation $\mathfrak{t}$. Then there can be no arc $(i,j)$ in $\mathfrak{t}$ with $i < k < j$, as it would intersect with infinitely many arcs incident with $k$ (of the form $(k,l)$ with $l > j$ if $k$ is a right-fountain, or $(m,k)$ with $m < i$ if $l$ is a left-fountain). This means that there are two infinite connected components in an exchange quiver corresponding to a triangulation with fountain, one for the left-hand side and one for the right-hand side of the fountain. In the split fountain case with a left-fountain $l$ and a right-fountain $r$ we get an additional finite component for the arcs between the two fountains, i.e.\ all arcs $(i,j)$ with $l \leq i < j \leq r$.
} \qedhere
\end{itemize}
\end{proof}

\vspace{1em}
We can be even more precise about the forms of these exchange quivers for the different triangulations, our first main result.

\pagebreak

\begin{theorem}\label{classif-of-exch-quivers}
For the three types of triangulations of the $\infty$-gon, we get the following types of exchange quivers.
\begin{itemize}
\item[(1)]{A locally finite triangulation $\mathfrak{t}$ gives rise to an exchange quiver $Q_{\mathfrak{t}}$ of the following form:
\begin{itemize}
\item[i)]{The quiver $Q_{\mathfrak{t}}$ is connected.}
\item[ii)]{It has a subquiver of type $A_{+\infty}$ and each such subquiver is not linearly oriented.}
\item[iii)]{There are copies of quivers of finite $A$-mutation type (i.e.\ mutation-equivalent to a quiver of type $A_n$ for some finite $n$) glued to the subquiver of type $A_{+\infty}$ via oriented $3$-cycles.}
\end{itemize}
}
\item[(2)]{A triangulation $\mathfrak{t}$ with a fountain gives rise to an exchange quiver $Q_{\mathfrak{t}}$ of the following form:
\begin{itemize}
\item[i)]{The quiver $Q_{\mathfrak{t}}$ has two connected components.}
\item[ii)]{Each connected component has a subquiver of type $A_{+\infty}$ with a linear orientation.}
\item[iii)]{There are copies of quivers of finite $A$-mutation type glued to the subquivers of type $A_{+\infty}$ via oriented $3$-cycles.}
\end{itemize}
}
\item[(3)]{A triangulation $\mathfrak{t}$ with a split fountain gives rise to an exchange quiver $Q_{\mathfrak{t}}$ of the following form:
\begin{itemize}
\item[i)]{The quiver $Q_{\mathfrak{t}}$ has three connected components, two of them infinite and one finite.}
\item[ii)]{Each infinite connected component has a subquiver of type $A_{+\infty}$ with a linear orientation. The finite component is of finite $A$-mutation type.}
\item[iii)]{There are copies of quivers of finite $A$-mutation type glued to the subquivers of type $A_{+\infty}$ via oriented $3$-cycles.}
\end{itemize}
}
\end{itemize}
\end{theorem}

\begin{proof}

Consider first a locally finite triangulation $\mathfrak{t}$. Starting with any arc $(i_0,j_0)$ in $\mathfrak{t}$, we get a sequence of arcs $\{(i_k,j_k)\}_{k \in \mathbb{Z}}$ such that $(i_k,j_k)$ is the minimal arc passing over $(i_{k-1},j_{k-1})$ for all $k \geq 1$.
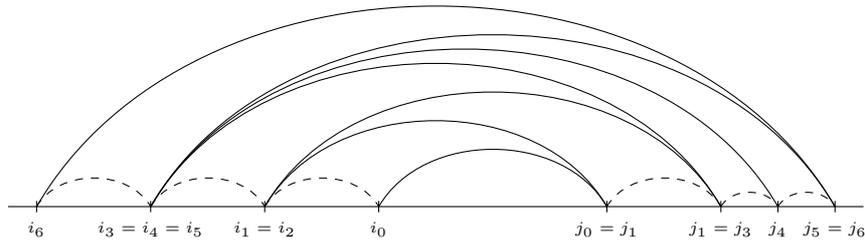
\begin{center}
\begin{tikzpicture}[scale =.75]
\tikzstyle{every node}=[font=\tiny]
\draw (-9.5,0) -- (5.5,0);
\draw (-3,0.1) -- (-3,-0.1) node[below]{$i_0$};
\draw (1,0.1) -- (1,-0.1) node[below]{$j_0 = j_1$};
\draw (-5,0.1) -- (-5,-0.1) node[below]{$i_1 = i_2$};
\draw (3,0.1) -- (3,-0.1) node[below]{$j_1 = j_3$};
\draw (-7,0.1) -- (-7,-0.1) node[below]{$i_3 = i_4 = i_5$};
\draw (4,0.1) -- (4,-0.1) node[below]{$j_4$};
\draw (5,0.1) -- (5,-0.1) node[below]{$j_5 = j_6$};
\draw (-9,0.1) -- (-9,-0.1) node[below]{$i_6$};

\path (-3,0) edge [out= 60, in= 120] (1,0);
\path (-5,0) edge [out= 60, in= 120] (1,0);
\path (-5,0) edge [out= 60, in= 120] (3,0);
\path (-7,0) edge [out= 60, in= 120] (3,0);
\path (-7,0) edge [out= 60, in= 120] (4,0);
\path (-7,0) edge [out= 60, in= 120] (5,0);
\path (-9,0) edge [out= 60, in= 120] (5,0);
\path[dashed] (1,0) edge [out= 60, in= 120] (3,0);
\path[dashed] (-5,0) edge [out= 60, in= 120] (-3,0);
\path[dashed] (-7,0) edge [out= 60, in= 120] (-5,0);
\path[dashed] (-9,0) edge [out= 60, in= 120] (-7,0);
\path[dashed] (4,0) edge [out= 60, in= 120] (5,0);
\path[dashed] (3,0) edge [out= 60, in= 120] (4,0);
\end{tikzpicture}
\end{center}
This yields a subquiver with underlying graph 
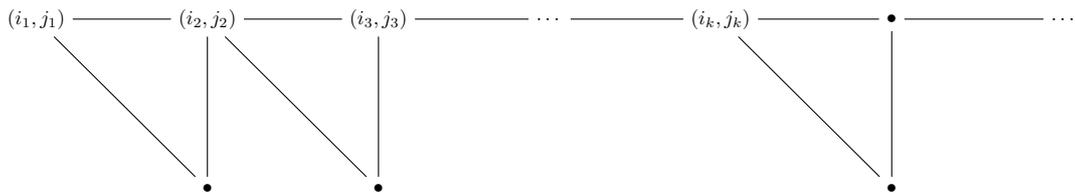
\begin{center}
\begin{tikzpicture}[scale = .75, transform shape]
\tikzstyle{every node}=[font=\small]
\node (1) at (0,0) {$(i_1,j_1)$};
\node (2) at (3,0) {$(i_2,j_2)$};
\node (3) at (6,0) {$(i_3,j_3)$};

\node (4) at (3,-3) {$\bullet$};
\node (5) at (6,-3) {$\bullet$};

\node (6) at (9,0) {\ldots};

\node (7) at (12,0) {$(i_k,j_k)$};
\node (8) at (15,0) {$\bullet$};
\node (9) at (15,-3) {$\bullet$};

\node (10) at (18,0) {$\ldots$};

\draw (1) -- (2);
\draw (3) -- (2);
\draw (4) -- (2);
\draw (1) -- (4);

\draw (5) -- (2);
\draw (3) -- (5);
\draw (3) -- (6);

\draw (6) -- (7);
\draw (8) -- (7);
\draw (9) -- (7);
\draw (8) -- (9);

\draw (8) -- (10);
\end{tikzpicture}
\end{center}
where each $3$-cycle containing the vertices $(i_{k-1},j_{k-1})$ and $(i_k,j_k)$ in the graph has a clockwise orientation whenever $(i_k,j_k)$ passes over $(i_{k-1},j_{k-1})$ to the right, and an anti-clockwise orientation whenever $(i_k, j_k)$ passes over $(i_{k-1},j_{k-1})$ to the left. Note that we cannot have a linear orientation on the subquiver of type $A_{+\infty}$ generated by the sequence $\{(i_k,j_k)\}_{k \in \mathbb{Z}}$ --- else there would exist a $k$ such that $(i_{k'},j_{k'})$ passes over  $(i_{k'-1},j_{k'-1})$ to the right (respectively to the left) for all $k'>k$ and hence $i_{k'} = i_k$ (respectively $j_{k'} = j_k$) for all $k>k'$ and we would have a right- (respectively a left-) fountain.

Furthermore, we have to consider the arcs and sides $(m,n)$ in the gaps, i.e.\ the arcs $(m,n)$ with $i_k \leq m < n \leq i_{k-1}$, respectively $j_{k-1} \leq m <n \leq j_k$. Consider all the arcs in such a gap. We assume  $(i_k,j_k)$ passes over $(i_{k-1},j_{k-1})$ to the right; analogous considerations apply to the case where it passes over $(i_{k-1},j_{k-1})$ to the left. Then the arc $(j_{k-1},j_k)$ passes over all other arcs in the gap. 
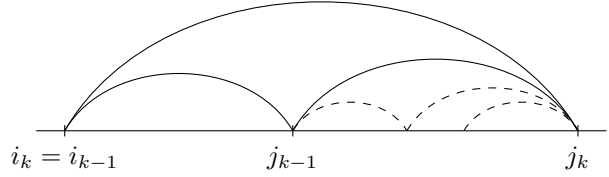
\begin{center}
\begin{tikzpicture}[scale =.75]
\tikzstyle{every node}=[font=\small]
\draw (-3.5,0) -- (6.5,0);
\draw (-3,0.1) -- (-3,-0.1) node[below]{$i_k = i_{k-1}$};
\draw (1,0.1) -- (1,-0.1) node[below]{$j_{k-1}$};
\draw (6,0.1) -- (6,-0.1) node[below]{$j_k$};

\path (-3,0) edge [out= 60, in= 120] (1,0);
\path (-3,0) edge [out= 60, in= 120] (6,0);
\path (1,0) edge [out= 60, in= 120] (6,0);
\path[dashed] (3,0) edge [out= 60, in= 120] (6,0);
\path[dashed] (1,0) edge [out= 60, in= 120] (3,0);
\path[dashed] (4,0) edge [out= 60, in= 120] (6,0);
\end{tikzpicture}
\end{center}
We can view the set of arcs in the gap as a triangulation of a $(|j_{k-1}-j_k|+1)$-gon, where $(j_{k-1},j_k)$ plays the role of an additional side. Hence the subquiver $Q_k$ corresponding to this subset of the triangulation is mutation-equivalent to a quiver of type $A_{(|j_{k-1}-j_k|-2)}$ with frozen vertices corresponding to the sides $(j_{k-1},j_{k-1}+1),\; \ldots,\;(j_k-1,j_k)$. In place of the last ``side'' $(j_{k-1},j_k)$ of the $(|j_{k-1}-j_k|-2)$-gon we get the mutable vertex $(j_{k-1},j_k)$ which is connected to the type $A_{+\infty}$ subquiver:
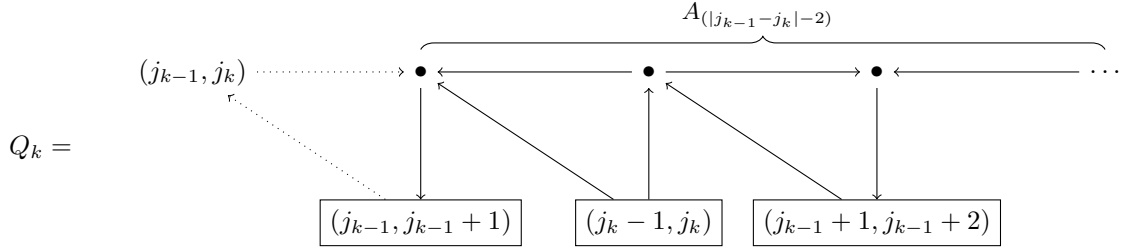
\begin{center}
\begin{tikzpicture}
\tikzstyle{every node}=[font=\small]
\node (a) at (0,2) {$(j_{k-1},j_k)$};
\node (b) at (3,2) {$\bullet$};
\node (c) at (6,2) {$\bullet$};
\node (d) at (9,2) {$\bullet$};
\node (e) at (12,2) {$\ldots$};

\node[draw, shape=rectangle] (f) at (3,0) {$(j_{k-1},j_{k-1}+1)$};
\node[draw, shape=rectangle] (g) at (6,0) {$(j_k-1,j_k)$};
\node[draw, shape=rectangle] (h) at (9,0) {$(j_{k-1}+1,j_{k-1}+2)$};

\draw[->, dotted] (a) -- (b);
\draw[->, dotted] (f) -- (a);
\draw[->] (c) -- (b);
\draw[->] (b) -- (f);
\draw[->] (g) -- (c);
\draw[->] (g) -- (b);
\draw[->] (c) -- (d);
\draw[->] (h) -- (c);
\draw[->] (e) -- (d);
\draw[->] (d) -- (h);

\draw [decorate,decoration={brace,amplitude=5pt}, yshift=-20pt]
   (3,3)  -- (12,3) 
   node [black,midway,above=4pt,xshift=-2pt] {\footnotesize $A_{(|j_{k-1}-j_k|-2)}$};
\node (x) at (-2,1) {$Q_k=$};

\end{tikzpicture}
\end{center}
Note that the quiver $Q_k$ is attached to the vertex $(j_{k-1},j_k)$ by a $3$-cycle. The exchange quiver for a locally finite triangulation is thus of the form
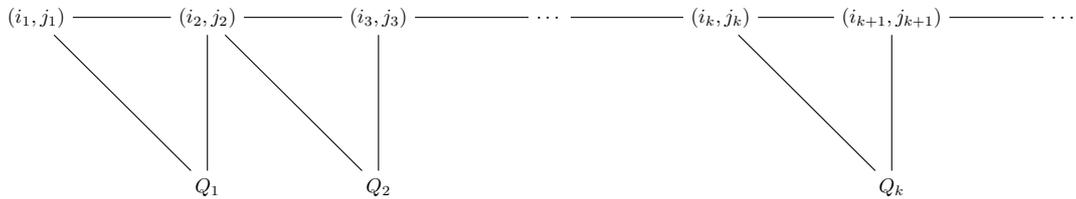
\begin{center}
\begin{tikzpicture}[scale = .75, transform shape]
\tikzstyle{every node}=[font=\small]
\node (1) at (0,0) {$(i_1,j_1)$};
\node (2) at (3,0) {$(i_2,j_2)$};
\node (3) at (6,0) {$(i_3,j_3)$};

\node (4) at (3,-3) {$Q_1$};
\node (5) at (6,-3) {$Q_2$};

\node (6) at (9,0) {\ldots};

\node (7) at (12,0) {$(i_k,j_k)$};
\node (8) at (15,0) {$(i_{k+1},j_{k+1})$};
\node (9) at (15,-3) {$Q_k$};

\node (10) at (18,0) {$\ldots$};

\draw (1) -- (2);
\draw (3) -- (2);
\draw (4) -- (2);
\draw (1) -- (4);

\draw (5) -- (2);
\draw (3) -- (5);
\draw (3) -- (6);

\draw (6) -- (7);
\draw (8) -- (7);
\draw (9) -- (7);
\draw (8) -- (9);

\draw (8) -- (10);
\end{tikzpicture}
\end{center}
If for $k \geq 1$ the arc $(i_k,j_k)$ passes over $(i_{k-1},j_{k-1})$ to the right the oriented 3-cycle containing $(i_{k-1},j_{k-1})$ and $(i_k,j_k)$ has a clockwise orientation and $Q_k$ is mutation-equivalent to a quiver of type $A_{(|j_{k-1}-j_k|-2)}$ with frozen vertices corresponding to the sides. If on the other hand  the arc $(i_k,j_k)$ passes over $(i_{k-1},j_{k-1})$ to the left, the oriented 3-cycle containing $(i_{k-1},j_{k-1})$ and $(i_k,j_k)$ has an anti-clockwise orientation  and $Q_k$ is mutation-equivalent to a quiver of type $A_{(|i_{k}-i_{k-1}|-2)}$ with frozen vertices.

The considerations for triangulations with a fountain or a split fountain are very similar to the locally finite case. However, we cannot just start with any arc and build a sequence of minimal arcs passing over each other, as we would never arrive at the left-fountain if we had started with an arc in the right-fountain, and vice versa. In fact, we need to build a sequence of arcs including arcs from both the right-fountain and the left-fountain. Let $l$ be a left-fountain and $r$ a right-fountain in a given triangulation, where $r$ and $l$ might coincide. We have a generic place to start for both the right and the left-fountain---namely the ``smallest fountain arcs'', i.e.\ the arcs $(r,s_0)$ and $(k_0,l)$ where $s_0$ is minimal and $k_0$ is maximal among all arcs of these forms in the triangulation. This gives us sequences of arcs $\{(r,s_i)\}_{i \in \mathbb{Z}_{\geq 0}}$ and $\{(k_i,l)\}_{i \in \mathbb{Z}_{\geq 0}}$ where $(r,s_{i})$ is the minimal arc passing over $(r,s_{i-1})$ and $(k_i,l)$ is the minimal arc passing over $(k_{i-1},l)$. We call the sequence $\{(r,s_i),(k_i,l)\}_{i \in \mathbb{Z}}$ the sequence of fountain arcs.

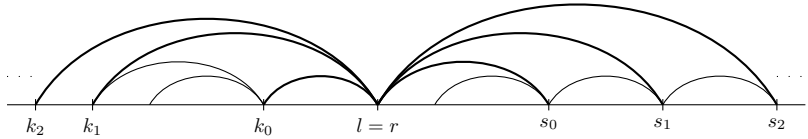
\begin{center}
\begin{tikzpicture}[scale = .75, transform shape]
\tikzstyle{every node}=[font=\small]
\draw (-6.5,0) -- (7.5,0);
\draw (0, 0.1) -- (0, -0.1) node [below] {$l=r$};
\draw (3, 0.1) -- (3, -0.1) node [below] {$s_0$};
\draw (5, 0.1) -- (5, -0.1) node [below] {$s_1$};
\draw (7, 0.1) -- (7, -0.1) node [below] {$s_2$};
\draw (-2, 0.1) -- (-2, -0.1) node [below] {$k_0$};
\draw (-5, 0.1) -- (-5, -0.1) node [below] {$k_1$};
\draw (-6, 0.1) -- (-6, -0.1) node [below] {$k_2$};
\draw[loosely dotted] (-6.5,0.5) -- (-6,0.5);
\draw[loosely dotted] (7,0.5) -- (7.5,0.5);
\path[thick] (0,0) edge [out= 60, in= 120] (5,0);
\path[thick] (0,0) edge [out= 60, in= 120] (3,0);
\path (1,0) edge [out= 60, in= 120] (3,0);
\path (3,0) edge [out= 60, in= 120] (5,0);
\path[thick] (-5,0) edge [out= 60, in= 120] (0,0);
\path (-4,0) edge [out= 60, in= 120] (-2,0);
\path[thick] (-2,0) edge [out= 60, in= 120] (0,0);
\path (-5,0) edge [out= 60, in= 120] (-2,0);
\path[thick] (0,0) edge [out= 60, in= 120] (7,0);
\path (5,0) edge [out= 60, in= 120] (7,0);
\path[thick] (-6,0) edge [out= 60, in= 120] (0,0);
\end{tikzpicture}
\end{center}
Again, we can view the arcs in the gaps between $s_{i-1}$ and $s_i$, respectively between $k_i$ and $k_{i-1}$, as triangulations of a $(|s_i-s_{i-1}|+1)$-gon, respectively a $(|k_{i-1}-k_i|+1)$-gon. Hence for a triangulation with a fountain or a split fountain, the component of the exchange quiver corresponding to the right-fountain is of the form

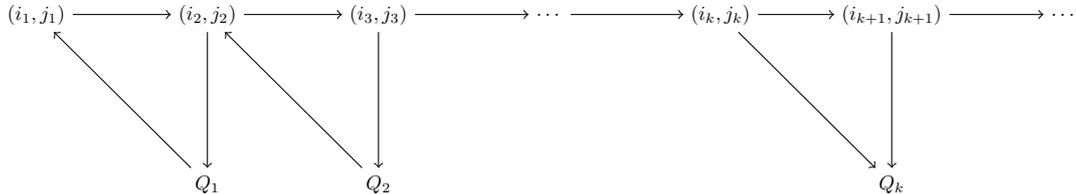
\begin{center}
\begin{tikzpicture}[scale = .75, transform shape]
\tikzstyle{every node}=[font=\small]
\node (1) at (0,0) {$(i_1,j_1)$};
\node (2) at (3,0) {$(i_2,j_2)$};
\node (3) at (6,0) {$(i_3,j_3)$};

\node (4) at (3,-3) {$Q_1$};
\node (5) at (6,-3) {$Q_2$};

\node (6) at (9,0) {\ldots};

\node (7) at (12,0) {$(i_k,j_k)$};
\node (8) at (15,0) {$(i_{k+1},j_{k+1})$};
\node (9) at (15,-3) {$Q_k$};

\node (10) at (18,0) {$\ldots$};

\draw[->] (1) -- (2);
\draw[<-] (3) -- (2);
\draw[<-] (4) -- (2);
\draw[<-] (1) -- (4);

\draw[->] (5) -- (2);
\draw[->] (3) -- (5);
\draw[->] (3) -- (6);

\draw[->] (6) -- (7);
\draw[<-] (8) -- (7);
\draw[<-] (9) -- (7);
\draw[->] (8) -- (9);

\draw[->] (8) -- (10);

\end{tikzpicture}
\end{center}
where for all $k\geq1$, $Q_k$ is mutation-equivalent to a quiver of type $A_{|j_k-j_{k-1}|-2}$ with frozen vertices corresponding to sides and the component of corresponding to the left-fountain is of a similar form but with anti-clockwise oriented 3-cycles containing $(i_{k-1},j_{k-1})$ and $(i_k,j_k)$ and with $Q_k$ mutation-equivalent to a quiver of type $A_{|i_{k-1}-i_{k}|-2}$.

These are all the components for the fountain case, as we included all the arcs in the right- as well as in the left-fountain. In the split fountain, we get an additional finite component for the gap between the left-fountain $l$ and the right-fountain $r$. As the arc $(l,r)$ is not mutable, we view it as a frozen vertex. Then the arcs in the gap between $l$ and $r$ are a triangulation of the $(|r-l| + 1)$-gon and the corresponding subquiver is mutation-equivalent to $A_{|r-l|+1}$ with frozen vertices corresponding to the sides $(l,l+1), \ldots, (r-1,r)$ and the frozen arc $(l,r)$.
\end{proof}

In the case of the locally finite triangulations and the triangulations with a fountain---that is, those corresponding to cluster tilting subcategories---the principal parts of the exchange quivers are actually the Gabriel quivers for their corresponding subcategories.   That is, the algorithm we have given to construct exchange quivers from triangulations of the $\infty$-gon do match the cluster structure on the category $D$.

\begin{theorem}\label{exchange quivers are Gabriel quivers}
Let $\mathfrak{t}$ be a triangulation of the $\infty$-gon which is locally finite or has a fountain. Then its exchange quiver $Q_{\mathfrak{t}}$ is an ice quiver which has as principal part the quiver of the cluster tilting subcategory corresponding to the triangulation $\mathfrak{t}$.
\end{theorem}

For the proof of Theorem \ref{exchange quivers are Gabriel quivers} we use the following lemma.  Recall that the shift functor $\Sigma$ acts (on coordinates) as $\Sigma(m,n)=(m-1,n-1)$.

\begin{lm}[{\cite[Proposition 2.2 and Lemma 3.6]{HJ}}]\label{non crossing arcs}
Let $x$ and $y$ be two indecomposable objects in the category $D$. Then we have $\Hom_D(x,y)=0$ if and only if the arcs corresponding to $\Sigma x$ and $y$ do not cross. Otherwise, we have $\Hom_D(x,y) \cong k$, where $k$ is the underlying field. \qed
\end{lm}

We now prove Theorem \ref{exchange quivers are Gabriel quivers}.

\begin{proof}
By construction, the vertices in the principal part of $Q_{\mathfrak{t}}$ correspond to indecomposable objects in a cluster tilting subcategory; we can match the frozen vertices to the zero object. Thus We only have to concern ourselves with the arrows. The exchange quiver $Q_{\mathfrak{t}}$ is composed of $3$-cycles of the form

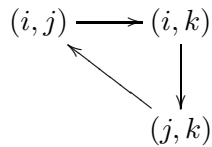
\begin{displaymath}
\xymatrix{(i,j) \ar[r] 	& (i,k) \ar[d]\\
					 	& (j,k) \ar[lu]}
\end{displaymath}
which come from triangles
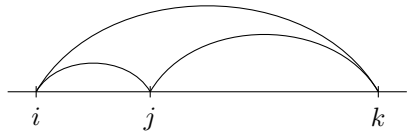
\begin{center}
\begin{tikzpicture}[scale =.75]
\tikzstyle{every node}=[font=\small]
\draw (-5.5,0) -- (1.5,0);
\draw (-3,0.1) -- (-3,-0.1) node[below]{$j$};
\draw (1,0.1) -- (1,-0.1) node[below]{$k$};
\draw (-5,0.1) -- (-5,-0.1) node[below]{$i$};

\path (-3,0) edge [out= 60, in= 120] (1,0);
\path (-5,0) edge [out= 60, in= 120] (1,0);
\path (-5,0) edge [out= 60, in= 120] (-3,0);

\end{tikzpicture}
\end{center}

The pairs of integers $(i,j)$, $(i,k)$ and $(j,k)$ label the vertices in the exchange quiver as well as the indecomposable objects of $D$, using the standard coordinates on the Auslander--Reiten quiver. Since we only make a claim on the principal part of $Q_{\mathfrak{}t}$, we assume all of $(i,j)$, $(i,k)$ and $(j,k)$ to correspond to non-zero objects in $D$. By Lemma \ref{non crossing arcs} we get that $\Hom_D((i,j),(i,k)) \cong k$, since $\Sigma(i,j)=(i-1,j-1)$ crosses $(i,k)$.  Similarly, $\Hom_D((i,k),(j,k)) \cong k$ and $\Hom((j,k),(i,j)) \cong k$ and all other morphism spaces between any of these objects are zero.

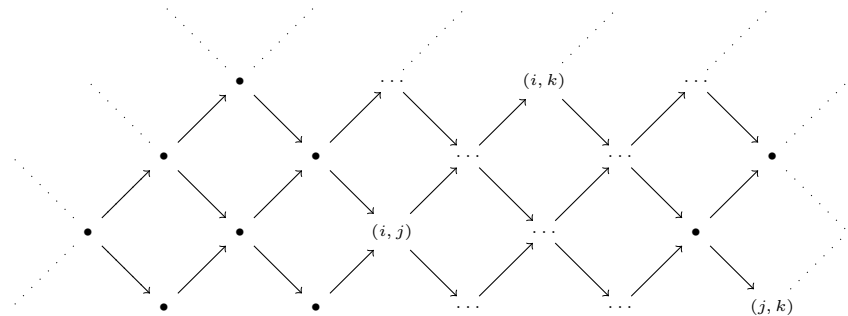
\begin{center}
\begin{tikzpicture}
\tikzstyle{every node}=[font=\tiny]
\node (a) at (0,0) {$\bullet$};
\node (b) at (1,1) {$\bullet$};
\node (c) at (2,2) {$\bullet$};
\node (d) at (3,3) {$\ldots$};
\node (e) at (2,0) {$\bullet$};
\node (f) at (3,1) {$(i,j)$};
\node (g) at (4,2) {$\ldots$};
\node (h) at (5,3) {$(i,k)$};
\node (i) at (4,0) {$\ldots$};
\node (j) at (5,1) {$\ldots$};
\node (k) at (6,2) {$\ldots$};
\node (l) at (7,3) {$\ldots$};
\node (m) at (6,0) {$\ldots$};
\node (n) at (7,1) {$\bullet$};
\node (o) at (8,2) {$\bullet$};
\node (p) at (8,0) {$(j,k)$};
\node (q) at (0,2) {$\bullet$};
\node (r) at (1,3) {$\bullet$};
\node (s) at (-1,1) {$\bullet$};

\draw[loosely dotted] (q) -- (s);
\draw[loosely dotted] (a) -- (s);
\draw[loosely dotted] (s) -- (-2,2);
\draw[loosely dotted] (s) -- (-2,0);
\draw[loosely dotted] (r) -- (0,4);
\draw[loosely dotted] (q) -- (-1,3);
\draw[loosely dotted] (o) -- (9,3);
\draw[loosely dotted] (o) -- (9,1);
\draw[loosely dotted] (p) -- (9,1);
\draw[loosely dotted] (l) -- (8,4);
\draw[loosely dotted] (d) -- (4,4);
\draw[loosely dotted] (h) -- (6,4);
\draw[loosely dotted] (r) -- (2,4);

\draw[->] (a) -- (b);
\draw[->] (b) -- (c);
\draw[->] (c) -- (d);
\draw[->] (e) -- (f);
\draw[->] (f) -- (g);
\draw[->] (g) -- (h);
\draw[->] (i) -- (j);
\draw[->] (j) -- (k);
\draw[->] (k) -- (l);
\draw[->] (m) -- (n);
\draw[->] (n) -- (o);
\draw[->] (q) -- (r);

\draw[->] (q) -- (b);
\draw[->] (r) -- (c);
\draw[->] (b) -- (e);
\draw[->] (c) -- (f);
\draw[->] (d) -- (g);
\draw[->] (f) -- (i);
\draw[->] (g) -- (j);
\draw[->] (h) -- (k);
\draw[->] (j) -- (m);
\draw[->] (k) -- (n);
\draw[->] (n) -- (p);
\draw[->] (l) -- (o);

\draw[->] (s) -- (a);
\draw[->] (s) -- (q);
\end{tikzpicture}
\end{center}
It remains to show that these maps do not factor through any other maps in the subcategory, so that they indeed define arrows in the Gabriel quiver. Let $(r,s)$ and $(t,u)$ be any two indecomposables in the cluster tilting subcategory $\T$ corresponding to the triangulation $\mathfrak{t}$. Then again by Lemma \ref{non crossing arcs} there are morphisms from $(r,s)$ to $(t,u)$ if and only if the arcs corresponding to $\Sigma(r,s) = (r-1,s-1)$ and $(t,u)$ intersect. Hence either we have
$$
r-1 < t < s-1 < u
$$
or
$$
t<r-1<u<s-1.
$$
On the other hand, because both $(t,u)$ and $(r,s)$ are in $\mathfrak{t}$ they cannot intersect so we get either
$
r \leq t < u \leq s
$,
$
t \leq r < s \leq u
$,
or
$
t < u \leq r < s
$. Hence only one of the following three cases is possible:

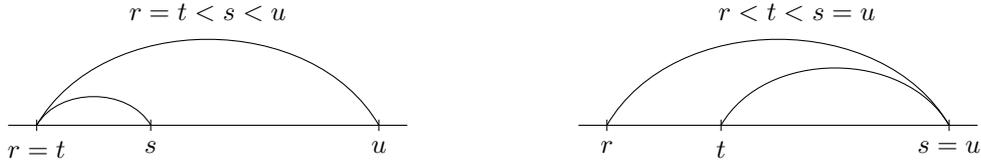
\begin{center}
\begin{tikzpicture}[scale =.75]
\tikzstyle{every node}=[font=\small]
\draw (-5.5,0) -- (1.5,0);
\draw (-3,0.1) -- (-3,-0.1) node[below]{$s$};
\draw (1,0.1) -- (1,-0.1) node[below]{$u$};
\draw (-5,0.1) -- (-5,-0.1) node[below]{$r=t$};

\node (Schrift1) at (-2,2) {$r=t<s<u$};

\path (-5,0) edge [out= 60, in= 120] (1,0);
\path (-5,0) edge [out= 60, in= 120] (-3,0);

\draw[xshift = 10cm] (-5.5,0) -- (1.5,0);
\draw[xshift = 10cm] (-3,0.1) -- (-3,-0.1) node[below]{$t$};
\draw[xshift = 10cm] (1,0.1) -- (1,-0.1) node[below]{$s=u$};
\draw[xshift = 10cm] (-5,0.1) -- (-5,-0.1) node[below]{$r$};

\node[xshift = 10cm] (Schrift1) at (-5,2) {$r<t<s=u$};

\path[xshift = 10cm] (-5,0) edge [out= 60, in= 120] (1,0);
\path[xshift = 10cm] (-3,0) edge [out= 60, in= 120] (1,0);
\end{tikzpicture}
\end{center}
or
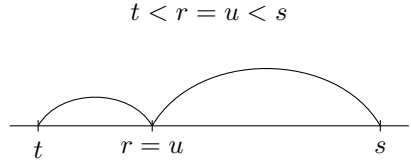
\begin{center}
\begin{tikzpicture}[scale =.75]
\tikzstyle{every node}=[font=\small]
\draw (-5.5,0) -- (1.5,0);
\draw (-3,0.1) -- (-3,-0.1) node[below]{$r=u$};
\draw (1,0.1) -- (1,-0.1) node[below]{$s$};
\draw (-5,0.1) -- (-5,-0.1) node[below]{$t$};

\node (Schrift) at (-2,2) {$t<r=u<s$};

\path (-3,0) edge [out= 60, in= 120] (1,0);
\path (-5,0) edge [out= 60, in= 120] (-3,0);

\end{tikzpicture}
\end{center}

Consider now again our triangle
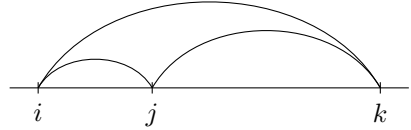
\begin{center}
\begin{tikzpicture}[scale =.75]
\tikzstyle{every node}=[font=\small]
\draw (-5.5,0) -- (1.5,0);
\draw (-3,0.1) -- (-3,-0.1) node[below]{$j$};
\draw (1,0.1) -- (1,-0.1) node[below]{$k$};
\draw (-5,0.1) -- (-5,-0.1) node[below]{$i$};

\path (-3,0) edge [out= 60, in= 120] (1,0);
\path (-5,0) edge [out= 60, in= 120] (1,0);
\path (-5,0) edge [out= 60, in= 120] (-3,0);

\end{tikzpicture}
\end{center}
Then the only objects $(m,n)$ such that there are maps $(i,j) \to (m,n)$ and $(m,n) \to (i,k)$ are $(m,n) = (i,j)$ or $(m,n)=(i,k)$ themselves, hence in particular the maps $(i,j) \to (i,k)$ do not factor through any other objects in the cluster tilting subcategory. The same argument holds for maps $(i,k) \to (j,k)$ and $(j,k) \to (i,k)$.  Therefore the arrows in $Q_{\mathfrak{t}}$ are in bijection with those in the corresponding Gabriel quiver and we are done.
\end{proof}

\subsection{Mutation equivalence classes}

Recall that a key feature of our definition of a cluster algebra of infinite rank is that we only allow finite sequences of mutations and hence not all the exchange quivers obtained in the previous section are mutation-equivalent.  Indeed the next theorem shows that there are very many equivalence classes.

\begin{theorem}
There are uncountably many mutation-equivalence classes of exchange quivers of locally finite triangulations of the $\infty$-gon, as well as of triangulations with a fountain or a split fountain.
\end{theorem}

\begin{proof}
First, a triangulation with a fountain and a locally finite triangulation can never be in the same mutation equivalence class, as it is never possible to dissolve a fountain in just finitely many steps. This is also clear from the classification of their exchange quivers (Theorem~\ref{classif-of-exch-quivers}), as quiver mutation preserves connectedness.

Any triangulation $\mathfrak{t}$ with a fountain can be characterised by its sequence of fountain arcs and the triangulations in the gaps, i.e. the sequence $\{(l_i,k),\mathfrak{t}_i\}_{i \in \mathbb{Z}>0} \cup \{(k,l_i),\mathfrak{t}_i\}_{i \in \mathbb{Z}>0}$, where $\mathfrak{t}_i$ is the triangulation of the $(|l_i-l_{i-1}|+1)$-gon occurring in the gap between $l_i$ and $l_{i-1}$ (where we set $l_0=k$) and $\{(l_i,k),(k,l_i),\}_{i \in \mathbb{Z}}$ is the sequence of fountain arcs.

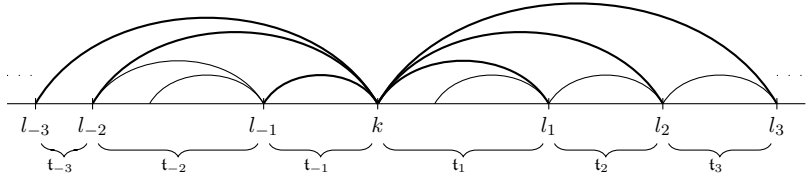
\begin{center}
\begin{tikzpicture}[scale = .75, transform shape]
%\tikzstyle{every node}=[font=\small]
\draw (-6.5,0) -- (7.5,0);
\draw (0, 0.1) -- (0, -0.1) node [below] {$k$};
\draw (3, 0.1) -- (3, -0.1) node [below] {$l_1$};
\draw (5, 0.1) -- (5, -0.1) node [below] {$l_2$};
\draw (7, 0.1) -- (7, -0.1) node [below] {$l_3$};
\draw (-2, 0.1) -- (-2, -0.1) node [below] {$l_{-1}$};
\draw (-5, 0.1) -- (-5, -0.1) node [below] {$l_{-2}$};
\draw (-6, 0.1) -- (-6, -0.1) node [below] {$l_{-3}$};
\draw[loosely dotted] (-6.5,0.5) -- (-6,0.5);
\draw[loosely dotted] (7,0.5) -- (7.5,0.5);
\path[thick] (0,0) edge [out= 60, in= 120] (5,0);
\path[thick] (0,0) edge [out= 60, in= 120] (3,0);
\path (1,0) edge [out= 60, in= 120] (3,0);
\path (3,0) edge [out= 60, in= 120] (5,0);
\path[thick] (-5,0) edge [out= 60, in= 120] (0,0);
\path (-4,0) edge [out= 60, in= 120] (-2,0);
\path[thick] (-2,0) edge [out= 60, in= 120] (0,0);
\path (-5,0) edge [out= 60, in= 120] (-2,0);
\path[thick] (0,0) edge [out= 60, in= 120] (7,0);
\path (5,0) edge [out= 60, in= 120] (7,0);
\path[thick] (-6,0) edge [out= 60, in= 120] (0,0);

\draw [decorate,decoration={brace,amplitude=5pt}, yshift=-20pt]
   (2.9,0)  -- (0.1,0) 
   node [black,midway,below=4pt,xshift=-2pt] {\footnotesize $\mathfrak{t}_1$};
\draw [decorate,decoration={brace,amplitude=5pt}, yshift=-20pt]
   (4.9,0)  -- (3.1,0) 
   node [black,midway,below=4pt,xshift=-2pt] {\footnotesize $\mathfrak{t}_2$};
\draw [decorate,decoration={brace,amplitude=5pt}, yshift=-20pt]
   (6.9,0)  -- (5.1,0) 
   node [black,midway,below=4pt,xshift=-2pt] {\footnotesize $\mathfrak{t}_3$};
\draw [decorate,decoration={brace,amplitude=5pt}, yshift=-20pt]
   (-0.1,0)  -- (-1.9,0) 
   node [black,midway,below=4pt,xshift=-2pt] {\footnotesize $\mathfrak{t}_{-1}$};
\draw [decorate,decoration={brace,amplitude=5pt}, yshift=-20pt]
   (-2.1,0)  -- (-4.9,0) 
   node [black,midway,below=4pt,xshift=-2pt] {\footnotesize $\mathfrak{t}_{-2}$};
\draw [decorate,decoration={brace,amplitude=5pt}, yshift=-20pt]
   (-5.1,0)  -- (-5.9,0) 
   node [black,midway,below=4pt,xshift=-2pt] {\footnotesize $\mathfrak{t}_{-3}$};
\end{tikzpicture}
\end{center}
Two triangulations with fountains characterised by the sequences 
$$
\{(l^1_i,k^1),\mathfrak{t}^1_i\}_{i \in \mathbb{Z}>0} \cup \{(k^1,l^1_i),\mathfrak{t}^1_i\}_{i \in \mathbb{Z}>0},
$$
respectively 
$$
\{(l^2_i,k^2),\mathfrak{t}_i\}_{i \in \mathbb{Z}>0} \cup \{(k^2,l^2_i),\mathfrak{t}^2_i\}_{i \in \mathbb{Z}>0},
$$ 
are mutation-equivalent if and only if they have the same fountain $k^1=k^2$ and there exist integers $m$ and $n$, such that for all $|j| \geq n$, 
$$
l^1_j = l^2_{j+n}
$$ 
and 
$$
\mathfrak{t}^1_j = \mathfrak{t}^2_{j+n}.
$$
This yields uncountably many mutation-equivalence classes. 

In a similar way, a locally finite triangulation $\mathfrak{t}$ is characterised by a sequence $\{(m_i,n_i),\mathfrak{t}_i\}_{i \in \mathbb{Z}\geq 0}$, where $(m_i,n_i)$ is the minimal arcs passing over $(m_{i-1},n_{i-1})$ for all $i \geq 1$ and $(m_0,n_0)$ is an arbitrary arc in $\mathfrak{t}$. Furthermore, for every $i \geq 1$, $\mathfrak{t}_i$ is the triangulation of the $(\max\{|m_{i-1}-m_i|,|n_i - n_{i-1}|\}+1)$-gon occurring in the gaps between the arcs of the sequence, and $\mathfrak{t}_0$ is the triangulation of the $(n_0-m_0)$-gon occurring underneath the arc $(m_0,n_0)$. 
Then two locally finite triangulations characterised by
$$
\{(m^1_i,n^1_i),\mathfrak{t}^1_i\}_{i \in \mathbb{Z}\geq 0},
$$
respectively
$$
\{(m^2_i,n^2_i),\mathfrak{t}^2_i\}_{i \in \mathbb{Z}\geq 0}
$$
are mutation-equivalent if and only if there exist integers $m$ and $n$, such that for all $|j| \geq n$, 
\begin{align*}
m^1_j = m^2_{j+n}\\
n^1_j = n^2_{j+n}
\end{align*}
and 
$$
\mathfrak{t}^1_j = \mathfrak{t}^2_{j+n}.
$$
 Note that this characterization of mutation-equivalence classes of locally finite triangulations does not depend on the first arc $(m_0,n_0)$ in the sequence. If we start with another arc $(m'_0,n'_0)$, by the proof of i) in Proposition~\ref{number of components}, there is an arc $(r,s)$ passing over both $(m_0,n_0)$ and $(m'_0,n'_0)$. However, any arc $(r,s)$ passing over both $(m_0,n_0)$ and $(m'_0,n'_0)$ appears in the sequence $\{(m_i,n_i)\}$ as well as in the sequence $\{(m'_i,n'_i)\}$. 

Let $\{(j_i,l_i)\}$ be a sequence of minimal arcs passing over each other. As the arcs in the sequence $\{(j_i,l_i)\}$ get arbitrarily wide (i.e.\ $j_i$ becomes arbitrarily negative and $l_i$ arbitrarily positive), there is an integer $k$, such that $j_k \leq r < s \leq l_k$ and such that either $j_{k-1} \geq r$ or $l_{k-1} \leq s$. If $(j_k,l_k)$ is passing over $(j_{k-1},l_{k-1})$ to the right  we have $r \leq j_{k-1} = l_k \leq r$, which implies $j_k \leq s$, as $(j_k,l_k)$ and $(r,s)$ must not intersect, hence $(j_k,l_k) = (r,s)$. The case where $(j_k,l_k)$ is passing over $(j_{k-1},l_{k-1})$ to the left follows from reflection. As $(r,s)$ appears in both sequences of minimal arcs, the sequences $\{(m_i,n_i),\mathfrak{t}_i\}$ and $\{(m'_i,n'_i),\mathfrak{t}'_i\}$ are the same onwards from the point where $(r,s)$ appears.
\end{proof}

Each representative $Q_{\mathfrak{t}}$ of a mutation equivalence class of exchange quivers of triangulations of the $\infty$-gon gives rise to a cluster algebra $\mathcal{A}_{Q_{\mathfrak{t}}}$ of infinite rank. Hence rather than giving rise to just one cluster algebra, the cluster structure on the category $D$ in fact consists of uncountably many sub-cluster structures (by which we mean cluster structures induced by subfamilies of cluster tilting objects in $D$), each of which gives rise to a cluster algebra. 

The category $D$ categorifies the cluster algebra $\mathcal{A}_{Q_{\mathfrak{t}}}$ via the cluster tilting subcategories that can be reached by finitely many mutations from the original cluster tilting subcategory corresponding to $\mathfrak{t}$. In \cite{JP} J\o rgensen and Palu define a Caldero--Chapoton map for categories with cluster tilting subcategories with possibly infinitely many isomorphism classes of indecomposables. When studying the Caldero--Chapoton map for the category $D$, they show that it can indeed be defined on all of $D$ if one considers cluster tilting subcategories corresponding to locally finite triangulations. However, if one starts with a fountain, the map cannot be extended on those indecomposables which cannot be reached by mutation of the corresponding cluster tilting subcategories.  Algebraically, we see this reflected in Theorem~\ref{t:CAonInfGr} below, where the split fountain case does not arise in the analysis of J\o rgensen--Palu because split fountains correspond to weak cluster tilting subcategories that are not cluster tilting.

\subsection{Cluster algebra structures on infinite Grassmannians}

Next we note that one can construct an infinite analogue of the coordinate rings of Grassmannians of planes in complex $n$-space.  This is described in detail in the appendix below, kindly provided by Michael Groechenig.  We show that this doubly-infinite Grassmannian coordinate ring and certain subrings of it admit cluster algebra structures arising from the above cluster combinatorics. 

The complex vector space of formal power series in $t$ and $t^{-1}$, $\mathbb{C}[[t,t^{-1}]]$, is a natural choice of pro-finite-dimensional vector space having a natural topological basis $\{ t_{i} \mid i\in \ZZ \}$.  As noted in Corollary~\ref{a:infiniteplucker} below, the homogeneous coordinate ring of the doubly-infinite Grassmannian $\Gr(2,\mathbb{C}[[t,t^{-1}]])$ may be identified with the graded ring that is the quotient of $\mathbb{C}[\Delta^{ij} \mid i<j]$ by the short Pl\"{u}cker relations
$$
\Delta^{ij}\Delta^{kl} = \Delta^{ik}\Delta^{jl} + \Delta^{il}\Delta^{jk}.
$$
where $i<k<j<l$.  We denote this ring by $\mathbb{C}[\Gr(2,\pm \infty)]$.

\begin{remark} This doubly-infinite Grassmannian is different from others appearing in the literature (e.g.\ \cite{SatoSato}, \cite{PressleySegal}, \cite{FioresiHacon}), although it is in a similar spirit.  We do not want a topologically complete ring, however, hence the construction described in the appendix is more appropriate.
\end{remark}

Then we may combine all the previous results of this section to obtain the following theorem on cluster algebras associated to triangulations of the $\infty$-gon.

\vfill
\pagebreak

\begin{theorem}\label{t:CAonInfGr} {\ }
\begin{itemize}
\item{A locally finite triangulation yields a cluster algebra structure on the whole coordinate ring $\mathbb{C}[\Gr(2, \pm \infty)]$, where we view $\Delta^{i,i+1}$ for $i \in \mathbb{Z}$ as coefficients.}
\item{A triangulation with a fountain $k$ yields a cluster algebra structure on the subalgebra of $\mathbb{C}[\Gr(2, \pm \infty)]$ generated by the Pl\"ucker coordinates $\{\Delta^{ij} \mid i < j \leq k \text{ or } k \leq i < j\}$, where we view $\Delta^{i,i+1}$ for $i \in \mathbb{Z}$ as coefficients.}
\item{A triangulation with a split fountain with left-fountain at $l$ and right-fountain at $r$, yields a cluster algebra structure on the subalgebra of $\mathbb{C}[\Gr(2, \pm \infty)]$ generated by the Pl\"ucker coordinates $\{\Delta^{ij} \mid i < j \leq l \text{ or } l \leq i < j \leq r \text{ or } r \leq i < j\}$, where we view $\Delta^{i,i+1}$ for $i \in \mathbb{Z}$ as well as $\Delta^{lr}$ as coefficients.}
\end{itemize}
\end{theorem}

\begin{proof}
The existence of the cluster algebra structures follows directly from the 1-1 correspondence of arcs in the $\infty$-gon and Pl\"ucker coordinates in $\mathbb{C}[\Gr(2, \pm \infty)]$ and the compatibility of diagonal flips with short Pl\"{u}cker relations.  The claims on generating sets follow from consideration of the sets of arcs that may be obtained from an initial triangulation by finite sequences of mutations.  For example, one clearly cannot obtain arcs crossing a fountain but, by viewing any finite part of a triangulation of the $\infty$-gon as a triangulation of an $n$-gon for some suitable $n$, clearly one may obtain all other arcs as in the finite case. 
\end{proof}

An avenue we intend to explore further is that of the higher doubly-infinite Grassmannians $\mathbb{C}[\Gr(k,\pm \infty)]$, to give an infinite generalisation of the work of Scott (\cite{Scott}).  In particular, one would like an infinite version of the combinatorics of Postnikov diagrams.

%We also note that candidates for the categorification of the cluster algebra structures on all Grassmannians (i.e.\ including $k>2$) have been announced by Jensen, King and Su.  It is not clear how a limit of their categories for $k=2$ might relate to the category $D$ and the work herein but it is expected that there should be a natural relationship.

\section{Quantum cluster algebra structures from triangulations of the \texorpdfstring{$\infty$}{infinity}-gon}

In line with \cite{GL}, we can show that the cluster algebra structures obtained from triangulations of the $\infty$-gon have quantum analogues on subalgebras of a suitable quantum analogue of $\mathbb{C}[\Gr(2,\pm \infty)]$.  By a well-known result of Leclerc and Zelevinsky (\cite[Theorem 1.1]{LZ}) we may see that a set of quantum Pl\"ucker coordinates corresponding to a triangulation of the $\infty$-gon together with the quantum Pl\"ucker coordinates corresponding to the sides form a maximal set of quasi-commuting variables.  This is because quasi-commuting of a pair of variables is a ``local'' claim, involving only four specific indices, so that the formul\ae\ of \cite{LZ} that are stated in the finite setting may be applied; we give more details below.  Then initial seeds corresponding to triangulations are eligible to be quantum clusters provided the compatibility condition holds.  In order to obtain quantum cluster algebra structures we need to verify this compatibility and also check that the quantum cluster variables obtained by mutation are again elements of the quantum coordinate ring.  (It will turn out that they are in fact quantum Pl\"{u}cker coordinates.)

To define the two-row doubly-infinite quantum matrices and doubly-infinite quantum Grassmannian, we follow the same pattern as set out in Section~\ref{ss:QCAs}.

\begin{defn} The (two-row) doubly-infinite quantum matrix algebra $\mathbb{C}_{q}[M(2,\pm \infty)]$ is the $\mathbb{C}$-algebra generated by the set $\{ X_{ij} \mid 1\leq i \leq 2,\ j\in \ZZ \}$ subject to the quantum matrix relations as given in Section~\ref{ss:QCAs}, where the second indices may now take any integer values.
\end{defn}

\begin{defn} The (two-row) doubly-infinite quantum Grassmannian $\mathbb{C}_{q}[\Gr(2,\pm \infty)]$ is the subalgebra of the (two-row) doubly-infinite quantum matrix algebra $\mathbb{C}_{q}[M(2,\pm \infty)]$ generated by the quantum Pl\"{u}cker coordinates $\Delta_{q}^{ij}=X_{1i}X_{2j}-qX_{1j}X_{2i}$ for all $i,j \in \ZZ$ such that $i<j$.
\end{defn}

Note that $\mathbb{C}_{q}[\Gr(2,\pm \infty)]$ is a domain since the finite quantum Grassmannians are.

We next look at the initial data.  Each triangulation yields an infinite exchange matrix $B$ corresponding to the exchange quiver we computed in Section \ref{Cluster algebra structures from triangulations}. Let $(i,j)$ be a mutable arc in the triangulation and let $(m,n)$ be an arc or a side. Then it follows directly from the construction of the exchange quiver that the matrix $B$ has entries

\begin{align*}
B_{(m,n),(i,j)} & = 
\begin{cases} 
1 &	\text{ { \small if $(m,n)$ is the minimal arc passing over $(i,j)$ to the left or}} \\
 &	\text{ {\small $(i,j)$ is the minimal arc passing over $(m,n)$ to the right}} \\
-1 & \text{ {\small if $(m,n)$ is the minimal arc passing over $(i,j)$ to the right or}} \\
 & \text{ {\small $(i,j)$ is the minimal arc passing over $(m,n)$ to the left,}} \\
0 & \text{ {\small otherwise.} } \\
\end{cases}
\end{align*}

Note that if an arc $(l,r)$ is not mutable, such as can happen in the split fountain case, we treat it as a frozen vertex, so there is no column in $B$ corresponding to $(l,r)$, only a row.

Let $L$ be the infinite matrix consisting of entries $L_{(i,j),(k,l)}$ for any arcs or sides $(i,j)$ and $(k,l)$ in the triangulation, where
\begin{displaymath}
\Delta_q^{ij}\Delta_q^{kl} = q^{L_{(i,j),(k,l)}}\Delta_q^{kl}\Delta_q^{ij}.
\end{displaymath}
As noted above, we can compute the entries using Theorem 1.1 of \cite{LZ} and we obtain that
\begin{align*}
L_{(i,j),(k,l)} & = 
\begin{cases} 
0 &	\text{if } i = k < j = l \text{ or } i < k < l < j \text{ or } k < i < j < l\\
1 & \text{if } i < k < j = l \text{ or } i = k < j < l \text{ or } k < l = i < j\\
-1 &\text{if } k < i < j = l \text{ or } i = k < l < j \text{ or } i < j = k < l\\
2 &\text{if } i < j < k < l \\
-2 &\text{if } k < l < i < j. \\
\end{cases}
\end{align*}

Recall from Section~\ref{ss:QCAs} that the defining data for a quantum cluster algebra are required to satisfy a compatibility condition: we verify that this indeed holds.

\begin{prop}\label{compatible}
The pair of matrices $(B,L)$ is compatible.
\end{prop}

\begin{proof}
Let $\T$ be the triangulation of the $\infty$-gon, $B$ the corresponding exchange matrix and $L$ the matrix encoding the quasi-commutation relations. We show that for $(i,j)$ a mutable arc and $(k,l)$ an arc or a side in $\T$ we have $(B^TL)_{(i,j),(k,l)} = 2 \delta_{(i,j),(k,l)}$.  Then
\begin{align*}
(B^TL)_{(i,j),(k,l)} 	&= \sum_{(x,y) \in \T}B^T_{(i,j),(x,y)}L_{(x,y),(k,l)} \\
					&= \sum_{(x,y) \in \T}B_{(x,y),(i,j)}L_{(x,y),(k,l)}. 
\end{align*}
The entry $B_{(x,y),(i,j)}$ is non-zero if and only if either $(x,y)$ is the minimal arc passing over $(i,j)$ or vice versa. Furthermore, $(i,j)$ is the diagonal in a unique quadrilateral $(m,i,n,j)$ or $(i,m,j,n)$. In the matrix $B$, only the rows corresponding to the sides and the other diagonal in this quadrilateral have non-zero entries in the row corresponding to the arc $(i,j)$. 

\vfill
\pagebreak
Then we have
\begin{displaymath}
(B^TL)_{(i,j),(k,l)} = L_{(i,m),(k,l)} + L_{(j,n),(k,l)} - L_{(m,j),(k,l)} - L_{(i,n),(k,l)}
\end{displaymath}
if $i < m < j < n$
and
\begin{displaymath}
(B^TL)_{(i,j),(k,l)} = L_{(m,j),(k,l)} + L_{(i,n),(k,l)}  - L_{(n,j),(k,l)}
\end{displaymath}
if $m < i < n < j$.
Inserting the entries of $L$ for all possible respective positions of the arcs $(i,j)$ and $(k,l)$ yields the desired result.
\end{proof}

Therefore for a given triangulation $\T$, the triple $(\{\Delta_q^{ij}\}_{(i,j) \in \T},B, L)$ is a quantum seed and gives rise to a quantum cluster algebra structure, not necessarily on $\mathbb{C}_q[\Gr(2, \pm \infty)]$  itself but possibly on some localization of this.  If we can show that each quantum cluster variable is a quantum Pl\"ucker coordinate, this gives us a quantum cluster algebra structure on (a subalgebra of) $\mathbb{C}_q[\Gr(2, \pm \infty)]$.  (This is more than is needed to show that we do not need to localize and this is a feature of being in the $k=2$ case.  As shown in \cite{GL}, for $k>2$ the quantum cluster variables need not be Pl\"{u}cker coordinates.)

\begin{prop}\label{quantum coordinate}
Let $\T$ be a triangulation of the $\infty$-gon and let $(i,j)$ be an arc in $\T$. Then the variable $\mu(\Delta_q^{ij})$ obtained by mutation of the corresponding quantum cluster at $(i,j)$ is a quantum Pl\"ucker coordinate.
\end{prop}

\begin{proof}
The following proof follows the same lines as the corresponding result in \cite{GL}; the above combinatorial descriptions of $B$ and $L$ allow some slight simplification.  Let $(i,j)$ be the diagonal in the quadrilateral $(i,m,j,n)$ with $i < m < j < n$.
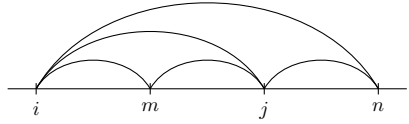
\begin{center}
\begin{tikzpicture}[scale=.75, transform shape]
\tikzstyle{every node}=[font=\small]
\draw (-5.5,0) -- (1.5,0);
\draw (1,0.1) -- (1,-0.1) node[below]{$n$};
\draw (-1,0.1) -- (-1,-0.1) node[below]{$j$};
\draw (-5,0.1) -- (-5,-0.1) node[below]{$i$};
\draw (-3,0.1) -- (-3,-0.1) node[below]{$m$};

\path (-5,0) edge [out= 60, in= 120] (-3,0);
\path (-1,0) edge [out= 60, in= 120] (1,0);
\path (-3,0) edge [out= 60, in= 120] (-1,0);
\path (-5,0) edge [out= 60, in= 120] (1,0);
\path (-5,0) edge [out= 60, in= 120] (-1,0);
\end{tikzpicture}
\end{center}
The case where $m<i<n<j$ follows from reflection. We compute the mutation of $\Delta_{q}^{ij}$, $\mu(\Delta_q^{ij})$, in the above context according to the quantum exchange relation (as described in Section~\ref{ss:QCAs}, but also indexing the standard basis vectors appearing in the monomials $M$ by integer pairs, as we have for $B$ and $L$).  
\begin{align*}
\mu(\Delta_q^{ij}) 	& = M(-e_{(i,j)} + \sum_{B_{(k,l),(i,j)}>0}B_{(k,l),(i,j)}e_{(k,l)}) + M(-e_{(i,j)} - \sum_{B_{(k,l),(i,j)}<0}B_{(k,l),(i,j)}e_{(k,l)}) \\
					& = M(-e_{(i,j)} + 1 \cdot e_{(i,m)} + 1 \cdot e_{(j,n)}) + M(-e_{(i,j)} - (-1)\cdot e_{(m,j)} - (-1)\cdot e_{(i,n)}) \\
					& = q^{\frac{1}{2}(1 \cdot (-1)\cdot L_{(i,j),(i,m)} + 1 \cdot 1\cdot L_{(j,n),(i,m)} + (-1) \cdot 1\cdot L_{(j,n),(i,j)})} \Delta_q^{im}(\Delta_q^{ij})^{-1}\Delta_q^{jn} \\
					& \qquad + q^{\frac{1}{2}((-1) \cdot 1\cdot L_{(i,n),(i,j)} + (-1) \cdot 1\cdot L_{(m,j),(i,j)} + 1 \cdot 1\cdot L_{(m,j),(i,n)})}(\Delta_q^{ij})^{-1}\Delta_q^{in}\Delta_q^{mj} \\
					& = q^{\frac{1}{2}((-1) \cdot (-1) + 1 \cdot (-2) + (-1) \cdot (-1)} \Delta_q^{im}(\Delta_q^{ij})^{-1}\Delta_q^{jn} \\
					& \qquad + q^{\frac{1}{2}((-1) \cdot (-1) + (-1) \cdot (-1) + 1 \cdot 0)}(\Delta_q^{ij})^{-1}\Delta_q^{in}\Delta_q^{mj}	\\						 	& = \Delta_q^{im}(\Delta_q^{ij})^{-1}\Delta_q^{jn} + q(\Delta_q^{ij})^{-1}\Delta_q^{in}\Delta_q^{mj}\\
					& = q^{-1}(\Delta_q^{ij})^{-1}\Delta_q^{im}\Delta_q^{jn} + q(\Delta_q^{ij})^{-1}\Delta_q^{in}\Delta_q^{mj}
\end{align*}
By the quantum Pl\"ucker relations we have
\begin{align*}
\mu(\Delta_q^{ij}) \Delta_q^{ij} & = q^{-1}\Delta_q^{im}\Delta_q^{jn} + q\Delta_q^{in}\Delta_q^{mj} \\
								& = \Delta_q^{mn} \Delta_q^{ij},
\end{align*}
and therefore, as $\mathbb{C}_{q}[\Gr(2,\pm \infty)]$ is a domain, we deduce that $\mu(\Delta_q^{ij}) = \Delta_q^{mn}$ and the former is indeed a quantum Pl\"ucker coordinate.
\end{proof}

Propositions \ref{compatible} and \ref{quantum coordinate} prove the following facts. 

\begin{cor} {\ }
\begin{itemize}
\item{A locally finite triangulation yields a quantum cluster algebra structure on the whole doubly-infinite quantum Grassmannian $\mathbb{C}_q[\Gr(2, \pm \infty)]$,  where we view $\Delta_{q}^{i,i+1}$ for $i \in \mathbb{Z}$ as coefficients.}
\item{A triangulation with a fountain $k$ yields a quantum cluster algebra structure on the \\ subalgebra of $\mathbb{C}_q[\Gr(2, \pm \infty)]$ generated by the quantum Pl\"ucker coordinates \hfill \hfill \linebreak $\{\Delta_q^{ij} \mid i < j \leq k \text{ or } k \leq i < j\}$,  where we view $\Delta_{q}^{i,i+1}$ for $i \in \mathbb{Z}$ as coefficients.}
\item{A triangulation with a split fountain with left-fountain at $l$ and right-fountain at $r$ yields a quantum cluster algebra structure on the subalgebra of $\mathbb{C}_q[\Gr(2, \pm \infty)]$ generated by the quantum Pl\"ucker coordinates $\{\Delta_q^{ij} \mid i < j \leq l \text{ or } l \leq i < j \leq r \text{ or } r \leq i < j\}$, where we view $\Delta_{q}^{i,i+1}$ for $i \in \mathbb{Z}$ as well as $\Delta_{q}^{lr}$ as coefficients.}
\end{itemize}

\end{cor}

Note that this provides the generalisation to the infinite rank case of the properties of the $n$-gon model for type $A$ cluster algebras and Grassmannians.  Namely, in the $k=2$ situation we are considering, every (quantum) cluster variable corresponds to an arc and to a (quantum) Pl\"{u}cker coordinate and every cluster corresponds to a triangulation.  That is, a maximal set of non-intersecting arcs corresponds to a maximal quasi-commuting set of quantum minors, whose indices form a maximal weakly separated set in the sense of Leclerc and Zelevinsky (\cite{LZ}).

Again, we expect that similar results hold for the higher doubly-infinite Grassmannians $\mathbb{C}_{q}[\Gr(k,\pm \infty)]$.

\appendix

\section{Grassmannians for infinite-dimensional vector spaces,\\ by Michael Groechenig}
In the following we fix an arbitrary field $k$ which will be the base field over which we perform all of our constructions. Nonetheless, to avoid awkward language we suppress the field from the notation; the reader should keep in mind that vector space will refer to $k$-vector space, algebra to $k$-algebra, scheme to $k$-scheme and so on.

Traditionally, one deals with infinite-dimensional objects arising in an algebraic framework by evoking the theory of ind-varieties. To a directed system of classical varieties $$\cdots \hookrightarrow V_i \subset V_{i+1} \hookrightarrow \cdots,$$ where the inclusions are closed immersions, one assigns the formal object $$V:=\bigcup_{i \in \mathbb{N}} V_i.$$
Here, varieties are understood to be reduced separated schemes of finite type. If the dimension of the $V_i$ is increasing, we think of $V$ as an infinite-dimensional geometric object. An important example of an ind-variety is given by ind-affine space $\mathbb{A}^{\mathrm{ind}}$. In order to introduce this ind-variety, one stipulates $V_i$ to be $\mathbb{A}^i := \Spec k[x_0,\dots,x_{i-1}]$, and $V_i \hookrightarrow V_{i+1}$ the map corresponding to the surjection of rings $k[x_0,\dots,x_i] \twoheadrightarrow k[x_0,\dots,x_{i-1}],$ which sends $x_j \mapsto x_j$ for $j \leq i-1$ and $x_i \mapsto 0$.

A regular function $V \to \mathbb{A}^{1}$ is given by a compatible system of regular functions $V_i \to \mathbb{A}^1$. We conclude that the coordinate ring $k[V]$ of $V$, i.e.\ the ring of regular functions, arises as an inverse limit along the maps $$\cdots \twoheadleftarrow k[V_{i+1}] \twoheadleftarrow k[V_i] \twoheadleftarrow \cdots.$$ In particular, $k[V]$ is a canonical way a topological ring.

For ind-affine space one obtains for the coordinate ring a certain topological completion of the polynomial ring in infinitely many generators $$k[\mathbb{A}^{\mathrm{ind}}] = \widehat{k[x_i \mid i \in \mathbb{N}]},$$ which also contains the infinite sum $\sum_{i=0}^{\infty} x_i.$
From a purely algebraic viewpoint it might seem more natural to define infinite-dimensional affine space $\mathbb{A}^{\infty}$ to be the affine scheme $\Spec k[x_i|i \in \mathbb{N}]$ associated to the polynomial ring in countably many generators. The coordinate ring of this object is by definition $$k[\mathbb{A}^{\infty}] = k[x_i \mid i \in \mathbb{N}].$$ 

In this appendix we will outline a scheme-theoretic alternative to the traditional ind-variety approach to infinite-dimensional Grassmannians. The homogeneous coordinate ring of the infinite-dimensional Grassmannian we construct will be the home for the cluster algebra structures described by the authors of the present article, and not a topological completion thereof.

\subsection{Grassmannians for finite-dimensional vector spaces}

Following \cite{Gro95} we adopt a functorial approach to Grassmannians. To a scheme $X$ we associate the functor from (commutative) algebras to sets $$X(-)\colon \Alg \to \Set.$$ For an algebra $A$ we define $X(A)$ to be the set of maps $\Spec A \to X$ and refer to it as the set of $A$-points of $X$. Vice versa, given such a functor, one might wonder whether it is represented by a scheme.
 
Let $V$ be a finite-dimensional vector space. The Grassmannian $\GrV$ is a scheme whose $k$-points $\GrV(k)$ are given by $n$-dimensional subspaces $U \subset V$. Dualizing this set-up we can also think of these subspaces in terms of quotients of the dual space $$V^{\vee} \twoheadrightarrow U^{\vee}.$$ We will denote the corresponding geometric object by $\Grd$; its $k$-points are understood to be $n$-dimensional quotients of $V^{\vee}$. But as we can always dualize we expect to have a natural isomorphism $$\GrV \cong \Grd.$$ In the definition below we define the functor corresponding to $\Grd$, and use the above isomorphism as a definition for $\GrV$.

\begin{defn}\label{defi:gr}
Let $V^{\vee}$ be a vector space.  We denote by $\Grd(A)$ the set of locally free $A$-modules $W$ of rank $n$ together with a surjection $$A \otimes V^{\vee} \twoheadrightarrow W.$$ The corresponding functor will be denoted by $$\GrV = \Grd: \Alg \to \Set.$$

If $n = 1$ we will denote the functor $\Grd$ also by $\Pd$.
\end{defn}

Note that it is not true that $\GrV(A)$ is the set of locally free subspaces $U$ of rank $n$ of $A \otimes V$. By dualizing one only obtains the subspaces $U$ such that the quotient $A \otimes V / U$ is locally free itself. For that reason it is slightly more convenient to define the functor $\GrV$ as above in terms of $\Grd$.

It is a well-known fact that $\GrV$ is a scheme, indeed this is a special case of Grothendieck's representability result for Quot-schemes (\cite[Theorem~3.1]{Gro95}). Nonetheless in this specific situation an easier proof can be given (which is again well-known); the proof which we present in the next subsection for an infinite-dimensional $V$ is merely a straightforward generalization of the known argument.

\subsection{The infinite-dimensional case}

We denote by $V^{\vee}$ a vector space of arbitrary possibly infinite dimension. Since $V^{\vee}$ can be realized as an ind-object in the category of finite-dimensional vector spaces, its dual $V$ should be naturally thought of as a pro-object in the category of finite-dimensional vector spaces. Alternatively one can view $V$ as a topological vector space by endowing it with the inverse limit topology. It is then true that $V^{\vee}$ is the topological dual of $V$. In particular we remind the reader that $n$-dimensional subspaces of $V$ correspond to $n$-dimensional quotients of $V^{\vee}$ by dualizing. 

Every pro-finite dimensional vector space $V$ gives rise to an affine scheme $$\mathbb{A}(V) := \Spec \Sym V^{\vee},$$ which represents the functor $\mathbb{A}(V)(A) := (A \widehat{\otimes} V)^{\mathrm{set}}$; in particular its $k$-points are given by the underlying set of the topological vector space $V$.

The objects $V$ and $V^{\vee}$ are examples of so-called Tate vector spaces (\cite{Dri06}). The category of Tate vector spaces can be defined as a certain subcategory of ind-objects in pro-finite-dimensional vector spaces (\cite{Pre07}).

Since Definition \ref{defi:gr} applies to arbitrary vector spaces we can study next whether the functor $\Grd$ defines a scheme.

\begin{lm}
The functor $\Grd$ is representable by a scheme.
\end{lm}

\begin{proof}
Let $(x_i)_{i \in I}$ be an ordered basis for $V^{\vee}$. We denote by $j$ an ordered $n$-tuple of $(x_i)_{i \in I}$, and by $U_j$ the subspace of $V^{\vee}$ spanned by the corresponding subset of the chosen basis. We consider now an algebra $A$ and an $A$-point $$(\phi: A \otimes V^{\vee} \twoheadrightarrow W) \in \Grd.$$ After refining $\Spec A$ by an affine covering we may assume without loss of generality that $W$ is equivalent to the free $A$-module $A^n$ of rank $n$. If we refine $\Spec A$ for a second time we may assume that there exists a $j$ such that $\phi|_{U_j}$ is an isomorphism. Locally on $\Spec A$ this allows us to compare this given $A$-point of the Grassmannian with the surjection $$\phi_{j,M}: A^I \twoheadrightarrow A^n,$$ which restricts to the identity map on $U_j$, and where $M \in \Hom(A^{I - j},A^I) \cong A^{(I-j) \times I}$ denotes an arbitrary matrix of infinite size. In particular we conclude that the functor $\Grd$ can be covered by the infinite-dimensional affine spaces $\mathbb{A}^{(I-j) \times I}$.
\end{proof}

\subsection{The Pl\"ucker embedding}

There is a natural morphism of schemes (i.e.\ natural transformation of functors) $$p: \Grd \to \Pdd,$$ which sends an $A$-point $\phi: A \otimes V^{\vee} \twoheadrightarrow W$ to $\wedge^n\phi: \bigwedge^nV^{\vee} \twoheadrightarrow \bigwedge^nW.$ There is a natural morphism $$\mathbb{A}(\bigwedge^nV^{\vee}) - 0 \to \Pdd.$$ To see this one observes that an $A$-point of the scheme on the left is given by a certain morphism of rings $\Sym \bigwedge^nV^{\vee} \to A$, which gives rise to a surjection $$A \otimes \bigwedge^nV^{\vee} \to A$$ when tensored with $A$. The image of the base change
\[
\xymatrix{
P \ar[r] \ar[d] & \mathbb{A}(\bigwedge^{\mathit{n}}V^{\vee}) - 0 \ar[d] \\
\Grd \ar[r] & \Pdd
}
\]
can be described by explicit equations, which are derived from the so-called Pl\"ucker relations. We refer to Chapter 1.5 in \cite{GriffithsHarris} for an exposition of the Pl\"ucker embedding for finite-dimensional vector spaces.

Recalling that $V$ is a pro-finite-dimensional vector space, we define $\bigwedge^n V$ and similar constructions in the obvious way, by using an explicit realization of $V$ as an inverse system of finite-dimensional vector spaces. For every $z \in \bigwedge^{n-1} V$ we have a contraction operator $$i(z): \bigwedge^n V^{\vee} \to V,$$ defined by the adjunction $(i(z)x,y) = (z,x \wedge y).$

\begin{lm}
The scheme $P$ is the closed subscheme of $\mathbb{A}(\bigwedge^nV^{\vee}) - 0$ given by the system of equations $$i(z)x \wedge x = 0,$$ where $z$ runs through the elements of $\bigwedge^{n-1} V$. 
\end{lm}

\begin{proof}
In Chapter 1.5 of \cite{GriffithsHarris} this statement is verified for a finite-dimensional vector space $V$. In general $V$ is a pro-object in finite-dimensional vector spaces, respectively $V^{\vee}$ is a vector space of arbitrary dimension (i.e.\ an ind-object in finite-dimensional vector spaces). Let us choose an explicit realization $V^{\vee} = \bigcup_{i \in I} V^{\vee}_i$ as ind-object. We see that to an $A$-point of the closed subscheme cut out by the Pl\"ucker relations gives rise to a compatible system of $n$-dimensional quotients of $A \otimes V^{\vee}_i$ (if $i$ is large enough). This gives rise to a unique rank $n$ quotient of $A \otimes V^{\vee}$. 
\end{proof}

In this paper $V$ is the pro-finite-dimensional vector space $k[[t,t^{-1}]]$, with topological basis $\{ t^i \mid i \in \mathbb{Z} \}$. Its topological dual is a discrete vector space with countable basis $\lambda_i$, where $i \in \mathbb{Z}$. The second exterior power $\bigwedge^2k[[t,t^{-1}]])$ has a topological basis formed by $t^i \wedge t^j$, where $i <  j$. The homogeneous coordinate ring of the corresponding infinite-dimensional projective space $\mathbb{P}(\bigwedge^2k[[t,t^{-1}]]^{\vee})$ is the graded ring $k[\Delta^{ij} \mid i <  j]$ with $\deg \Delta^{ij} = 1$. We think of $\Delta^{ij}$ as the coefficient of the element $t^i\wedge t^j$ of the above topological basis.

\begin{cor}\label{a:infiniteplucker}
Let $V$ be the pro-finite-dimensional vector space $k[[t,t^{-1}]]$. The Pl\"ucker embedding realizes $\mathrm{Gr}(2,V)$ as the subvariety of $\mathbb{P}(\bigwedge^2k[[t,t^{-1}]]^{\vee})$ given by the homogeneous polynomials $$ \Delta^{ij} \Delta^{kl} - \Delta^{ik}\Delta^{jl} + \Delta^{il}\Delta^{jk} = 0,$$ where $i  < k < j < l$.
\end{cor}

\begin{proof}
As above this is shown as in Chapter 1.5 of \cite{GriffithsHarris}. Since we have $i(z)x \wedge x = 0 = \frac{1}{2}i(z)(x \wedge x)$, and $n=2$, we conclude that the condition is equivalent to $x \wedge x = 0$. Representing $x$ as $\frac{1}{2}\sum_{i,j}\Delta^{ij}t^i\wedge t^j,$ where $\Delta^{ij} = - \Delta^{ji}$, we obtain the equations for the coefficients $$ \Delta^{ij} \Delta^{kl} = \Delta^{ik}\Delta^{jl} + \Delta^{il}\Delta^{jk},$$ where we can assume without loss of generality that $i  < k < j < l$.
\end{proof}

\small

\bibliographystyle{amsplain}
\bibliography{biblio}\label{references}

\providecommand{\bysame}{\leavevmode\hbox to3em{\hrulefill}\thinspace}
\providecommand{\MR}{\relax\ifhmode\unskip\space\fi MR }
% \MRhref is called by the amsart/book/proc definition of \MR.
\providecommand{\MRhref}[2]{%
  \href{http://www.ams.org/mathscinet-getitem?mr=#1}{#2}
}
\providecommand{\href}[2]{#2}
\begin{thebibliography}{10}

\bibitem{BZ}
Arkady Berenstein and Andrei Zelevinsky, \emph{Quantum cluster algebras}, Adv.
  Math. \textbf{195} (2005), no.~2, 405--455.

\bibitem{BIRS}
A.~B. Buan, O.~Iyama, I.~Reiten, and J.~Scott, \emph{Cluster structures for
  2-{C}alabi-{Y}au categories and unipotent groups}, Compos. Math. \textbf{145}
  (2009), no.~4, 1035--1079. \MR{2521253 (2010h:18021)}

\bibitem{BMRRT}
Aslak~Bakke Buan, Robert Marsh, Markus Reineke, Idun Reiten, and Gordana
  Todorov, \emph{Tilting theory and cluster combinatorics}, Adv. Math.
  \textbf{204} (2006), no.~2, 572--618.

\bibitem{Dri06}
Vladimir Drinfeld, \emph{Infinite-dimensional vector bundles in algebraic
  geometry: an introduction}, The unity of mathematics, Progr. Math., vol. 244,
  Birkh\"auser Boston, Boston, MA, 2006, pp.~263--304. \MR{2181808
  (2007d:14038)}

\bibitem{FioresiHacon}
R.~Fioresi and C.~Hacon, \emph{On infinite-dimensional {G}rassmannians and
  their quantum deformations}, Rend. Sem. Mat. Univ. Padova \textbf{111}
  (2004), 1--24. \MR{2076730 (2005h:58011)}

\bibitem{FZ1}
Sergey Fomin and Andrei Zelevinsky, \emph{Cluster algebras. {I}.
  {F}oundations}, J. Amer. Math. Soc. \textbf{15} (2002), no.~2, 497--529
  (electronic).

\bibitem{FZ2}
\bysame, \emph{Cluster algebras. {II}. {F}inite type classification}, Invent.
  Math. \textbf{154} (2003), no.~1, 63--121.

\bibitem{GSV}
Michael Gekhtman, Michael Shapiro, and Alek Vainshtein, \emph{Cluster algebras
  and {P}oisson geometry}, Mathematical Surveys and Monographs, vol. 167,
  American Mathematical Society, Providence, RI, 2010. \MR{2683456
  (2011k:13037)}

\bibitem{Gorsky}
Mikhail Gorsky, \emph{On {Y}oung diagrams, flips and cluster algebras of type
  {$A$}}, Preprint posted at arXiv:1106.2458.

\bibitem{GL}
Jan~E. Grabowski and St\'{e}phane Launois, \emph{Quantum cluster algebra
  structures on quantum {G}rassmannians and their quantum {S}chubert cells: the
  finite-type cases}, Int. Math. Res. Not. \textbf{2011} (2011), no.~10,
  2230--2262.

\bibitem{GriffithsHarris}
Phillip Griffiths and Joseph Harris, \emph{Principles of algebraic geometry},
  Wiley Classics Library, John Wiley \& Sons Inc., New York, 1994, Reprint of
  the 1978 original. \MR{1288523 (95d:14001)}

\bibitem{Gro95}
Alexander Grothendieck, \emph{Techniques de construction et th\'eor\`emes
  d'existence en g\'eom\'etrie alg\'ebrique. {IV}. {L}es sch\'emas de
  {H}ilbert}, S\'eminaire {B}ourbaki, {V}ol.\ 6, Soc. Math. France, Paris,
  1995, pp.~Exp.\ No.\ 221, 249--276. \MR{1611822}

\bibitem{HJ12}
Thorsten Holm and Peter J{\o}rgensen, \emph{Cluster tilting vs.\ weak cluster
  tilting in {D}ynkin type {$A$} infinity}, Forum Math., to appear. Preprint
  posted at arXiv:1201.3195.

\bibitem{HJ}
\bysame, \emph{On a cluster category of infinite {D}ynkin type, and the
  relation to triangulations of the infinity-gon}, Math. Z. \textbf{270}
  (2012), no.~1-2, 277--295. \MR{2875834}

\bibitem{HJR}
Thorsten Holm, Peter J{\o}rgensen, and Rapha{\"e}l Rouquier (eds.),
  \emph{Triangulated categories}, London Mathematical Society Lecture Note
  Series, vol. 375, Cambridge University Press, Cambridge, 2010. \MR{2723238
  (2011c:18001)}

\bibitem{JP}
Peter J{\o}rgensen and Yann Palu, \emph{A {C}aldero--{C}hapoton map for
  infinite clusters}, Trans. Amer. Math. Soc., to appear. Preprint posted at
  arXiv:1004.1343.

\bibitem{Ke10}
Bernhard Keller, \emph{Cluster algebras, quiver representations and
  triangulated categories}, Triangulated categories, London Math. Soc. Lecture
  Note Ser., vol. 375, Cambridge Univ. Press, Cambridge, 2010, pp.~76--160.
  \MR{2681708 (2011h:13033)}

\bibitem{KLR}
Ann Kelly, Tom~H Lenagan, and Laurent Rigal, \emph{Ring theoretic properties of
  quantum grassmannians}, J. Algebra Appl. \textbf{3} (2004), no.~1, 9--30.

\bibitem{LZ}
Bernard Leclerc and Andrei Zelevinsky, \emph{Quasicommuting families of quantum
  {P}l\"ucker coordinates}, Kirillov's seminar on representation theory, Amer.
  Math. Soc. Transl. Ser. 2, vol. 181, Amer. Math. Soc., Providence, RI, 1998,
  pp.~85--108.

\bibitem{PressleySegal}
Andrew Pressley and Graeme Segal, \emph{Loop groups}, Oxford Mathematical
  Monographs, The Clarendon Press Oxford University Press, New York, 1986,
  Oxford Science Publications. \MR{900587 (88i:22049)}

\bibitem{Pre07}
Luigi Previdi, \emph{Locally compact objects in exact categories}, Internat. J.
  Math. \textbf{22} (2011), no.~12, 1787--1821. \MR{2872533}

\bibitem{SatoSato}
Mikio Sato and Yasuko Sato, \emph{Soliton equations as dynamical systems on
  infinite-dimensional {G}rassmann manifold}, Nonlinear partial differential
  equations in applied science ({T}okyo, 1982), North-Holland Math. Stud.,
  vol.~81, North-Holland, Amsterdam, 1983, pp.~259--271. \MR{730247
  (86m:58072)}

\bibitem{Scott}
Joshua~S. Scott, \emph{Grassmannians and cluster algebras}, Proc. London Math.
  Soc. (3) \textbf{92} (2006), no.~2, 345--380.

\end{thebibliography}

\normalsize

\end{document}